\documentclass[11pt,letterpaper,reqno]{amsart}
\renewcommand{\baselinestretch}{1.1} 

\usepackage{epsfig}
\usepackage{amsmath}
\usepackage{amssymb}
\usepackage{amsthm}
\usepackage{indentfirst}
\usepackage{xspace}
\usepackage{multirow}
\usepackage{hyperref}
\usepackage{xcolor}
\definecolor{darkred}{rgb}{0.2,0.25,0.75}
\hypersetup{colorlinks=true,urlcolor=darkred,linkcolor=darkred,citecolor=darkred}
\usepackage{verbatim}
\usepackage[letterpaper,top=1.25in,left=1in,right=1in,bottom=1.25in]{geometry}
\usepackage{thmtools}
\usepackage[all]{xy}
\usepackage{longtable}
\usepackage{capt-of}
\usepackage{enumitem}
\usepackage{amsmath,amssymb,graphicx,wasysym,tikz}
 \usepackage[utf8]{inputenc}

\usetikzlibrary{arrows, positioning, shapes, decorations.pathreplacing, decorations.markings, calc, matrix}

\tikzset{->-/.style={decoration={
  markings,
  mark=at position #1 with {\arrow{>}}},postaction={decorate}}}

\usepackage{multicol}

 \usepackage{graphicx}
\usepackage{wrapfig}
\usepackage{lipsum}  % generates filler text

\def\impact#1{\bgroup\narrower%\footnotefont
\baselineskip
\footskip\bigbreak
\hrule\vspace{-0.1 in}\medskip\nobreak\noindent \begin{BI}
\renewcommand{\baselinestretch}{1.1}  {\it #1\/}\par\nobreak\end{BI}}
\def\endimpact{\medskip\nobreak \hrule\bigbreak\egroup}

\setlist{itemsep = 0.20em, topsep = 0.20em}

\declaretheoremstyle[spaceabove=0.25cm,spacebelow=0.25cm,notefont=\normalfont\bfseries, notebraces={(}{)}]{theorem}
\declaretheoremstyle[spaceabove=0.25cm,spacebelow=0.25cm,bodyfont=\normalfont,notefont=\normalfont\bfseries, notebraces={(}{)}]{noital}
\declaretheoremstyle[spaceabove=0.25cm,spacebelow=0.25cm,bodyfont=\normalfont\color{darkgreen},notefont=\normalfont\bfseries, notebraces={(}{)}]{green}
\declaretheoremstyle[spaceabove=0.25cm,spacebelow=0.25cm,bodyfont=\normalfont,notefont=\normalfont\bfseries,qed=$\qedsymbol$,notebraces={(}{)}]{proofstyle}

\usepackage{amsmath,thmtools}
\newtheorem{proposition}{Proposition}[]
\newtheorem{theorem}[proposition]{Theorem} 
\newtheorem{corollary}[proposition]{Corollary} 
\newtheorem{definition}[proposition]{Definition} 

\declaretheorem[name=Problem,style=theorem]{problem}
\newtheorem{BI}{Broader Impact}[]

\numberwithin{equation}{section}

\newcommand{\CC}{\mathbb{C}}

\newcommand{\e}{{\mathrm e}}

\newcommand{\git}{\mathbin{
  \mathchoice{/\mkern-6mu/}% \displaystyle
    {/\mkern-6mu/}% \textstyle
    {/\mkern-5mu/}% \scriptstyle
    {/\mkern-5mu/}}}% \scriptscriptstyle

\newcommand{\hkq}{\mathbin{
  \mathchoice{/\mkern-6mu/\mkern-6mu/}% \displaystyle
    {/\mkern-6mu/}% \textstyle
    {/\mkern-5mu/}% \scriptstyle
    {/\mkern-5mu/}}}% \scriptscriptstyle

\usepackage[framemethod=TikZ]{mdframed}
\usepackage{lipsum}
\mdfdefinestyle{MyFrame}{%
    linecolor=black,
    outerlinewidth=1pt,
    roundcorner=5pt,
    innertopmargin=\baselineskip,
    innerbottommargin=\baselineskip,
    innerrightmargin=10pt,
    innerleftmargin=10pt,
    backgroundcolor=gray!6!white}

\usepackage{blindtext}

\usepackage{tcolorbox}

\begin{document}

\setcounter{page}{1}

\title[Moduli Spaces of Generalized Hyperpolygons]{Moduli Spaces of Generalized Hyperpolygons}

\author{Steven Rayan}
\address{Steven Rayan: Centre for Quantum Topology and Its Applications (quanTA) and Department of Mathematics \& Statistics, University of Saskatchewan, Saskatoon, SK, S7N 5E6, Canada}
\email{rayan@math.usask.ca}

\author{Laura P.~ Schaposnik}
\address{Laura P.~ Schaposnik: Mathematics, Statistics, and Computer Science, University of Illinois at Chicago, 60607 Chicago, Illinois, USA}
\email{schapos@uic.edu}

\date{\today}
 
{\abstract{We introduce the notion of \emph{generalized hyperpolygon}, which arises as a representation, in the sense of Nakajima, of a comet-shaped quiver.  We identify these representations with rigid geometric figures, namely pairs of polygons: one in the Lie algebra of a compact group and the other in its complexification.  To such data, we associate an explicit meromorphic Higgs bundle on a genus-$g$ Riemann surface, where $g$ is the number of loops in the comet, thereby embedding the Nakajima quiver variety into a Hitchin system on a punctured genus-$g$ Riemann surface (generally with positive codimension).  We show that, under certain assumptions on flag types, the space of generalized hyperpolygons admits the structure of a completely integrable Hamiltonian system of Gelfand-Tsetlin type, inherited from the reduction of partial flag varieties.  In the case where all flags are complete, we present the Hamiltonians explictly. We also remark upon the discretization of the Hitchin equations given by hyperpolygons, the construction of triple branes (in the sense of Kapustin-Witten mirror symmetry), and dualities between tame and wild Hitchin systems (in the sense of Painlev\'e transcendents).}
}

\maketitle

\begin{center}\emph{In memory of Sir Michael Atiyah (1929-2019), an inspiration to geometers the world round.\\}\end{center}

\vspace{-0.1 in}
  
\tableofcontents

\section{Introduction}

One constant theme in the work of Michael Atiyah has been the interplay of algebra, geometry, and physics.  The construction of complete, asymptotically locally Euclidean (ALE), hyperk\"ahler $4$-manifolds --- in other words, of \emph{gravitational instantons} --- from a graph of Dynkin type is the capstone of a particular program for constructing K\"ahler-Einstein metrics, relevant to both geometry and physics and using only linear algebra.  This construction is at once the geometric realization of the McKay correspondence for finite subgroups of $\mbox{SU}(2)$ \cite{JM}, a generalization of the Gibbons-Hawking ansatz \cite{GH2}, and the analogue of the Atiyah-Drinfel'd-Hitchin-Manin technique \cite{ADHM} for constructing Yang-Mills instantons.  The construction completes a circle of ideas.  First, an instanton metric is determined, up to isometry and the integration of certain periods, by the metric on the tangent cone at infinity, as in Kronheimer \cite{PK}.  The metric data at infinity is given by a copy of $\mathbb C^2$ with the standard norm subjected to a Kleinian singularity.  The Kleinian singularity is produced by quotienting $\mathbb C^2$ by a finite subgroup $\Gamma<\mbox{SU}(2)$.  Finally, $\Gamma$ is determined by an affine ADE Dynkin diagram (or quiver).  Bringing this full circle, the quiver variety for the diagram, in an appropriate sense and with appropriate labels, is an instanton in the isometry class of the original one.

Historically, the Gibbons-Hawking ansatz (the type $A$ case of the above correspondences) and the ADHM method anctipate the Nakajima quiver variety construction \cite{HN}, which goes beyond instantons and $4$-manifolds.  The construction is a recipe for producing noncompact hyperk\"ahler varieties with arbitrarily large dimension from representation-theoretic data. Our present interest in these quiver varieties stems from some formal similarities between Nakajima quiver varieties and the \emph{Hitchin system}, which is an integrable system defined on the total space of the moduli space of semistable Higgs bundles on an algebraic curve.  The Hitchin system is a noncompact, complete hyperk\"ahler variety and, as with every Nakajima quiver variety, possesses an algebraic $\CC^\times$-action.  While Nakajima quiver varieties are finite-dimensional hyperk\"ahler quotients, the Hitchin system is an infinite-dimensional hyperk\"ahler one in the sense of \cite{MR877637}.   By some estimate, the Nakajima variety that comes ``closest'' to the parabolic Hitchin system in genus $0$ is the one arising from a star-shaped quiver, which is an object interlacing a number of $A$-type quivers.  This particular quiver variety can be regarded as a moduli space of so-called \emph{hyperpolygons}: this is both the hyperk\"ahler analogue of the moduli space of polygons studied, for instance, in \cite{MR1431002} and the ALE analogue of the moduli space of parabolic Higgs bundles at genus $0$.   Hyperpolygon spaces first appeared in \cite{MR1953356} and were studied from symplectic and toric points of view in \cite{MR2115372}.  

The connection between hyperpolygon spaces and rank-$2$ parabolic Higgs bundle moduli spaces at genus $0$ is initiated in the work of \cite{Ale,MR3424665}.  This was generalized to arbitrary rank in \cite{steve}, where the hyperpolygon space is explicitly identified with a degenerate locus of a corresponding Hitchin system.  They prove that this locus forms a sub-integrable system and furthermore show that the cohomology of the quiver variety has the hyperk\"ahler Kirwan surjectivity property (which was later established for all Nakajima quiver varieties in \cite{MR3773791}).  A compact version of the rank $3$ correspondence appears in \cite{MR3931789} while a version of this correspondence for logarithmic connections appears in \cite{MR3471271}.   A general overview of the relationship between hyperpolygons and Higgs bundles at genus $0$ is also provided in \cite{steveMFO}.

In this article, we extend this interaction between quiver varieties and Hitchin systems further by considering \emph{comet-shaped quivers}, which are star-shaped quivers with extra loops on the central vertex.  These are depicted in Fig. \ref{maps2}.  In addition to the extra loops, we allow arbitrary flags along each ``arm'' of the comet.

\begin{figure}[!h]
\begin{center}
\includegraphics[width=0.5\linewidth]{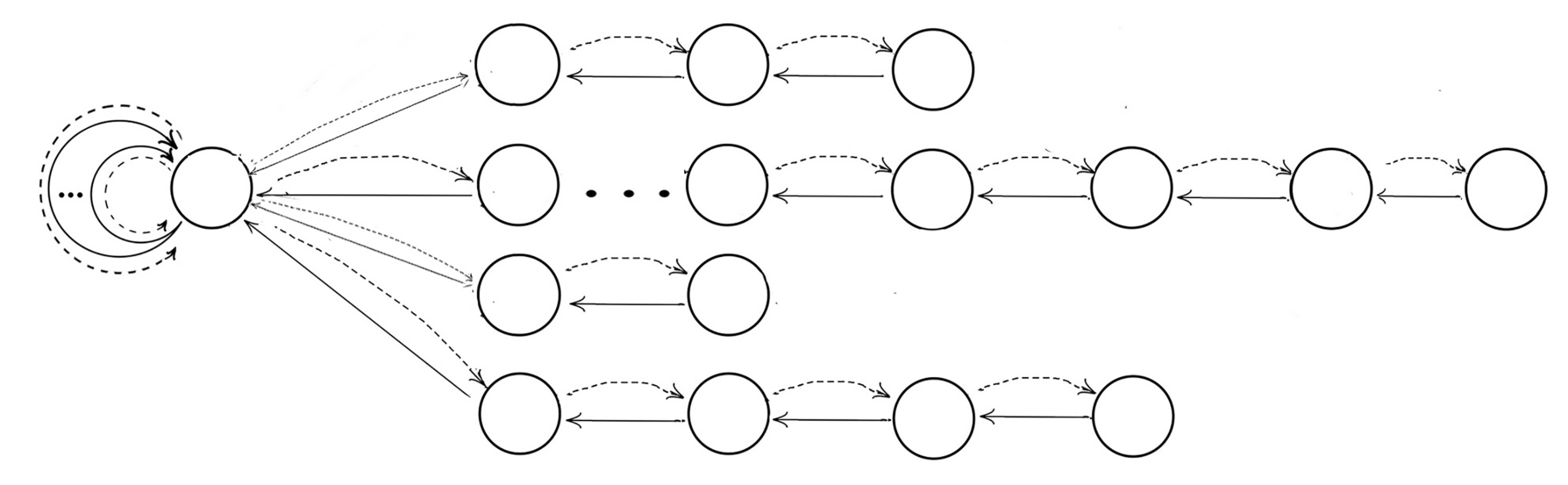}
\caption{A ``comet-shaped'' quiver. }\label{maps2}
\end{center}
\end{figure}

These quivers have been considered from a number of perspectives, such as in \cite{crawley2003matrices,hausel2011arithmetic}.  The work of \cite{hausel2011arithmetic} in particular suggests a close connection between Hitchin systems, character varieties, and comet-shaped quiver varieties.  The role of the extra loops is to increase the genus of the associated Hitchin system or character variety.   In this paper, we formalize this relationship.  After reviewing the construction of the Nakajima quiver variety associated to such a quiver, we describe classes in the Nakajima quiver variety associated to Fig. \ref{maps2} geometrically as \emph{generalized hyperpolygons}, which are pairs of polygons, one in the Lie algebra of a compact group and the other in its complexification. A Hamiltonian action of $\mbox{U}(1)$ is described, in analogy with the Hitchin system: it acts on one polygon while preserving the other, much as Higgs fields are rescaled while the holomorphic vector bundle is left invariant. We also remark upon how the real and complex moment map equations for generalized hyperpolygons are discrete analogues of the Hitchin equations, which to our knowledge has not been remarked upon in prior literature on hyperpolygons.  Motivated by these observations, we describe how to associate an explicit meromorphic Higgs bundle on a Riemann surface of genus $g$ to such a pair of polygons, where $g$ is the number of loops in the quiver, leading to an embedding map from the generalized hyperpolygon space into a moduli space of meromorphic, or equivalently twisted, Higgs bundles on a complex curve.  This is Theorem \ref{ThmEmbed}.  The image of the map generally has positive codimension and does not respect the hyperk\"ahler structures, as the Nakajima metric on the generalized hyperpolygon space is already complete.   We discuss a particular example, where the quiver has no loops and is the affine Dynkin diagram $\widetilde{\mbox{D}}_4$.  In this special case, the dimensions of the quiver variety and the associated Hitchin system are equal and the complement of the embedding is the Hitchin section.  

We then specialize to the case where each arm is either complete or minimal (where ``minimal'' means that the arm has exactly two nodes, an outer node and the central node, and the outer nodes are labelled ``$1$'').  In this case, we prove Theorem \ref{ThmIntegrable}: the generalized hyperpolygon space admits the structure of an algebraically-completely integrable Hamiltonian system of Gelfand-Tsetlin type.  We do this by appealing to the point of view that the quiver variety is a symplectic reduction of the product of cotangent bundles of partial flag varieties originating from the arms of the quiver, together with a product of contangent spaces of Lie algebras coming from the loop data.  Our argument relies essentially on being able to pass back and forth between the symplectic and geometric quotient.  As a result, we provide a maximal set of explicit, functionally-independent, Poisson-commuting Hamiltonians in the case where every arm is complete (Corollary \ref{CorPresent}).  This result is significant as it establishes the existence of explicit sub-integrable systems within parabolic Hitchin systems defined using only representation-theoretic data and not the complex structure on a Riemann surface. 

Finally, in Section \ref{final} we anticipdate two dualities involving hyperpolygons: mirror symmetry and a \emph{tame-wild} duality.  For the former, we consider the construction of triple branes in the moduli space of generalized hyperpolygons, as per the considerations of Kapustin-Witten \cite{Kap}.  For the latter, we discuss briefly an ambiguity between two types of comet quiver that leads to a passage from wild Higgs bundles to tame ones.  Both of these discussions anticipate further work.\\

\noindent\textbf{Acknowledgements.}  We thank D. Baraglia, I. Biswas, P. Crooks, J. Fisher, S. Gukov, T. Hausel, J. Kamnitzer, R. Mazzeo, H. Nakajima, A. Soibelman, J. Szmigielski, and H. Wei{\ss} for useful discussions related to this work.  S. Rayan is partially supported by an NSERC Discovery Grant, a Canadian Tri-Agency New Frontiers in Research Fund (Exploration) Grant, and a PIMS Collaborative Research Group (CRG) Grant.  L.P. Schaposnik is partially supported by the Humboldt Foundation, NSF Grant \#1509693, and NSF CAREER Award \#1749013.  Some formative steps in this work occurred during a research visit to the Fields Institute in May 2017 and during the Workshop on Singular Geometry and Higgs Bundles in String Theory at the American Institute of Mathematics in Fall 2017.  We appreciate the stimulating and productive research environments fostered by both institutes. Both authors are thankful to the Simons Center for Geometry and Physics for its hospitality during the Thematic Program on the Geometry and Physics of Hitchin Systems during January to June 2019, as well as to the Mathematisches Forschungsinstitut Oberwolfach where the authors had fruitful discussions with various participants of Workshop 1920. This material is also based upon work supported by the National Science Foundation under Grant No. DMS-1440140 while the authors were in residence at the Mathematical Sciences Research Institute in Berkeley, California, during the Fall 2019 semester.  Finally, we wish to acknowledge an anonymous referee for constructive feedback that led to improvements to the original manuscript.

\section{Review of Nakajima quiver varieties}\label{review}

The literature on Nakajima quiver varieties is, by now, more or less standard, although there are a few competing conventions.  We take a moment to establish ours.  We also feel it might be useful to establish the ``calculus'' of quiver moment maps --- that is, the rules out of which moment maps for group actions on the representation space of a quiver are built.

\subsection{Doubled quivers} Let $\mathcal Q$ be an ordinary quiver, meaning a directed graph with finitely-many vertices and directed edges.  We will use the terms ``vertex'' and ``node'' interchangeably; likewise, ``edge'' and ``arrow''.  We will denote the vertex set by $V(\mathcal Q)$ and the edge set by $E(\mathcal Q)$.  These data completely determine $\mathcal Q$ up to graph isomorphism.  We will permit multi-edges; that is, we will allow two or more edges to have the same tail and head.   If $u,v\in V(\mathcal Q)$, we will denote the set of edges with tail $u$ and head $v$ by $E_{u,v}(\mathcal Q)$.  Clearly, we have$$E(\mathcal Q)\;=\;\bigcup_{(u,v)\in V(\mathcal Q)\times V(\mathcal Q)}E_{u,v}(\mathcal Q)$$as a disjoint union.  The subset$$L(\mathcal Q)\;:=\;\bigcup_{u\in V(\mathcal Q)}E_{u,u}(\mathcal Q)$$ is the \emph{loop set} of $\mathcal Q$.  Finally, we shall denote by $\overline{\mathcal{Q}}$ the associated \emph{doubled} or \emph{Nakajima quiver}, which is fashioned from $\mathcal Q$ in the following way.  First, we define a set $D_{v,u}(\mathcal Q)$.  To begin, the set is empty.  Then, for each $e\in E_{u,v}(\mathcal Q)$, we create an element $-e$ of $D_{v,u}(\mathcal Q)$.  This element $-e$ will be given a graph-theoretic interpretation as an arrow from the node $v$ to the node $u$.  Next, we form the disjoint union $\overline{E_{v,u}}(\mathcal Q)=E_{v,u}(\mathcal Q)\cup D_{v,u}(\mathcal Q)$.  We subsequently define a graph $\overline{\mathcal Q}$ by setting $V(\overline{\mathcal Q})=V(\mathcal Q)$ and $E_{u,v}(\overline{\mathcal Q})=\overline{E_{u,v}}(\mathcal Q)$. A key feature of a Nakajima quiver is that we \emph{remember} which edges came from the original quiver, and so for each $(u,v)$ we have the distinguished subset $E_{u,v}(\mathcal Q)\subset E_{u,v}(\overline{\mathcal Q})$.  In other words, we remember which edges are ``$e$'' arrows and which are ``$-e$'' arrows.  (It is often the case that the original set $E_{u,v}(\mathcal Q)$ is empty while $E_{u,v}(\overline{\mathcal Q})$ is nonempty, e.g. when there is an arrow in $\mathcal Q$ pointing from $v$ to $u$ but none from $u$ to $v$.)  Typically the extra edges $-e$, called \emph{doubled edges}, are drawn \emph{dashed} to emphasize the original quiver. An example appears in Figure \ref{fig1}.  In particular, the set of loops of the Nakajima quiver, $L(\overline{\mathcal Q})$, has a distinguished subset $L(\mathcal Q)$ consisting of precisely the loops of the original quiver.

\begin{figure}[ht] 
 \begin{tikzpicture}
  \node[draw, circle] (Lsink) at (0,0) {$u_0$};
  \node[draw, circle] (Rsink) at (5,0) {$u_0$};
  
   \foreach \a in {45}
   \foreach \b in {90 }
     \foreach \c in {135 }
       \foreach \d in {180 }
     \foreach \e in {225 }
            \foreach \f in {270 }
     \foreach \g in {315 }
 {
    \node[draw, circle, scale=0.7] (Lv) at ($(Lsink)+(\a:1.75)$) {$u_n$};
        \path[-stealth] (Lv) edge (Lsink);
        \node[draw, circle, scale=0.7] (Lv) at ($(Lsink)+(\b:1.75)$) {$u_1$};
    \path[-stealth] (Lv) edge (Lsink);
            \node[draw, circle, scale=0.7] (Lv) at ($(Lsink)+(\c:1.75)$) {$u_2$};
    \path[-stealth] (Lv) edge (Lsink);
            \node[draw, circle, scale=0.7] (Lv) at ($(Lsink)+(\d:1.75)$) {$u_3$};
    \path[-stealth] (Lv) edge (Lsink);
            \node[draw, circle, scale=0.7] (Lv) at ($(Lsink)+(\e:1.75)$) {$u_4$};
    \path[-stealth] (Lv) edge (Lsink);
            \node[draw, circle, scale=0.7] (Lv) at ($(Lsink)+(\f:1.75)$) {$u_5$};
    \path[-stealth] (Lv) edge (Lsink);
           \node[draw, circle, scale=0.7] (Lv) at ($(Lsink)+(\g:1.75)$) {$u_6$};
    \path[-stealth] (Lv) edge (Lsink);

   \node[draw, circle, scale=0.7] (Rv) at ($(Rsink)+(\a:1.75)$) {$u_n$};
     \path[-stealth] (Rv) edge (Rsink);
        \path[-stealth, dashed] (Rsink) edge[bend left=20] (Rv) ;

        \node[draw, circle, scale=0.7] (Rv) at ($(Rsink)+(\b:1.75)$) {$u_1$};
     \path[-stealth] (Rv) edge (Rsink);
        \path[-stealth, dashed] (Rsink) edge[bend left=20] (Rv) ;

            \node[draw, circle, scale=0.7] (Rv) at ($(Rsink)+(\c:1.75)$) {$u_2$};
  \path[-stealth] (Rv) edge (Rsink);
        \path[-stealth, dashed] (Rsink) edge[bend left=20] (Rv) ;

            \node[draw, circle, scale=0.7] (Rv) at ($(Rsink)+(\d:1.75)$) {$u_3$};
  \path[-stealth] (Rv) edge (Rsink);
        \path[-stealth, dashed] (Rsink) edge[bend left=20] (Rv) ;

            \node[draw, circle, scale=0.7] (Rv) at ($(Rsink)+(\e:1.75)$) {$u_4$};
  \path[-stealth] (Rv) edge (Rsink);
        \path[-stealth, dashed] (Rsink) edge[bend left=20] (Rv) ;

            \node[draw, circle, scale=0.7] (Rv) at ($(Rsink)+(\f:1.75)$) {$u_5$};
  \path[-stealth] (Rv) edge (Rsink);
        \path[-stealth, dashed] (Rsink) edge[bend left=20] (Rv) ;

           \node[draw, circle, scale=0.7] (Rv) at ($(Rsink)+(\g:1.75)$) {$u_6$};
    \path[-stealth] (Rv) edge (Rsink);
        \path[-stealth, dashed] (Rsink) edge[bend left=20] (Rv) ;

  };
    \node (Ld) at ($(Lsink)+(0:1)$) {$\cdots$};
    \node (Rd) at ($(Rsink)+(0:1)$) {$\cdots$};

\end{tikzpicture}
\caption{Left, a quiver $\mathcal Q$ with nodes $u_i$; right, the doubled quiver $\overline{\mathcal{Q}}$ obtained from $\mathcal Q$.}
\label{fig1}
\end{figure}
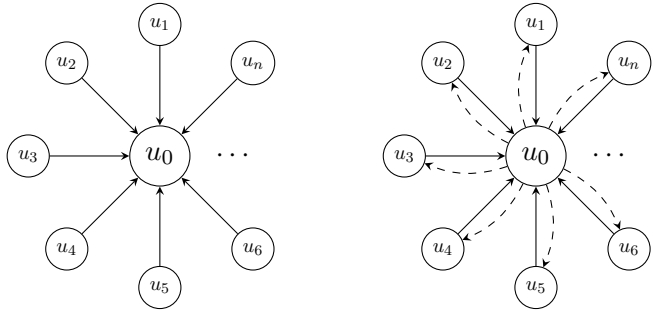

For our purposes, we need not only a quiver but also a set of labels for the quiver.  A labelling consists of a tuple $(r_u)_{u\in V(\mathcal Q)}$ where each $r_u$ is a positive integer. A \emph{representation} of $\mathcal Q$ is the choice of a linear map $x_{u,v}^e\in\mbox{Hom}(\CC^{r_u},\CC^{r_v})$ for each $u,v\in V(\mathcal Q)$ and each $e\in E_{u,v}(\mathcal Q)$. If the set $E_{u,v}(\mathcal Q)$ is empty, then we take $x=0\in\mbox{Hom}(\CC^{r_u},\CC^{r_v})$.   We put$$\mbox{Rep}_{u,v}(\mathcal Q)\;:=\;\mbox{Hom}(\CC^{r_u},\CC^{r_v})^{\oplus|E_{u,v}(\mathcal Q)|}.$$    The direct sum of these vector spaces for all pairs $(u,v)$ is denoted $\mbox{Rep}(Q)$, which we shall refer to as the \emph{space of representations} of $\mathcal Q$.   For $\overline{\mathcal Q}$, however, we do not construct its space of representations in this way.  Consider an edge $e\in E_{u,v}(\mathcal Q)\neq\varnothing$ and its corresponding doubled edge $-e\in E_{v,u}(\overline{\mathcal Q})$.  For $e$, we choose a representation $x^e_{u,v}\in\mbox{Rep}_{u,v}(\mathcal Q)$ as above.  On the other hand, for $-e$, we assign to it an element $y^{-e}_{v,u}$ of the cotangent space$$T^*_{x^e_{u,v}}\mbox{Rep}_{u,v}(\mathcal Q)=T^*_x\mbox{Hom}(\CC^{r_u},\CC^{r_v})\cong\mbox{Hom}(\CC^{r_u},\CC^{r_v})^*\cong\mbox{Hom}(\CC^{r_v},\CC^{r_u}).$$The right-most duality, given simply by the trace pairing, makes it clear that the ``direction'' of $y^{-e}_{v,u}$ as a map is consistent with the direction of the actual arrow given by $-e$.  If, on the other hand, we have $E_{u,v}(\mathcal Q)=\varnothing$, then $x=0$ and $y\in T^*_0\langle0\rangle$ and so $y=0$ too.

In general, a representation of $\overline{\mathcal Q}$ is a point$$(x,y):=(x^e_{u,v},y^{-e}_{v,u})_{(u,v)\in V(\mathcal Q)\times V(\mathcal Q),\,e\in E_{u,v}(\mathcal Q)}\in\mbox{Rep}(\overline{\mathcal Q})=T^*\mbox{Rep}(\mathcal Q).$$The main idea here is that the $y$ data (i.e. the representations of the doubled edges) are not independent of the $x$ data (i.e. the choices made for the original edges).  In some instances, we will write $y^{-e}_{v,u}(x)$ to emphasize the dependence. 

For our purposes we will declare a single node $\star\in V(\mathcal Q)$ to be the \emph{central node}.   This node will play a special role with regards to the construction of the so-called \emph{(Nakajima) quiver variety}, a (quasi)projective variety associated to the quiver and its labelling.  We will also impose the following rules: only the central node may have loops. In other words, $L(\mathcal Q)=L_{\star,\star}(\mathcal Q)$ (and likewise for $\overline{\mathcal Q}$).  Furthermore, we demand that the set $E_{\star,v}(\mathcal Q)$ is empty for every $v\neq\star$ in $V(\mathcal Q)$.  In other words, if $e$ is an edge in $\mathcal Q$ that is not a loop, then $\star$ cannot be the tail vertex of that edge.  We will typically use $a^\ell$ to denote a representation of a loop $\ell\in L(\mathcal Q)$ based at the central node and $b^{-\ell}$ for a representation of the corresponding doubled loop.  Furthermore, we impose the condition that $a^\ell$ is always a \emph{trace-free} element of $\mbox{Hom}(\CC^{r_\star},\CC^{r_\star})$, and so $a^\ell$ can be identified with an element of $\mathfrak{sl}(r_\star,\mathbb C)$ (and then so can $b^\ell$ by the duality between $\mathfrak{sl}(r_\star,\mathbb C)$ and its cotangent space at $a^\ell$).

\subsection{Calculus of quiver moment maps}

Consider now the decomposition of the vertex set $V(\mathcal Q)$ into $\Delta\cup\{s\}$.  We construct the group$$G=\left(\prod_{u\in\Delta}\mbox{U}(r_u)\times\mbox{SU}(r)\right)/\pm1,$$and denote by $\mathfrak g$ its Lie algebra. This group acts by change of basis in the expected way on $\mbox{Rep}(\mathcal Q)$: $\mbox{U}(r_u)$ and $\mbox{U}(r_v)$ act by multiplication on the right and the left, respectively, of $x^e_{u,v}$.   The component $\mbox{SU}(r)$ acts by multiplication on one side of representations of arrows entering or leaving the central node, and specifically by conjugation on any element of $L_{s,s}(\mathcal Q)$.  We regard $G$ as acting through its complexified adjoint action on pairs $(x^e_{u,v},y^{-e}_{v,u})$ in $\mbox{Rep}(\overline{\mathcal Q}):=T^*\mbox{Rep}(\mathcal Q)$.

The action by $G$ (respectively, by $G^{\mathbb C}$) is Hamiltonian with respect to the standard symplectic form on $\mbox{Rep}(\mathcal Q)$ (respectively, $\mbox{Rep}(\overline{\mathcal Q})$).  In particular, there are two moment maps that one can associate to the complex action.  These are the real moment map,
\begin{eqnarray}\mu:\mbox{Rep}(\overline{\mathcal{Q}})\longrightarrow\mathfrak{g}^*,\label{real}\end{eqnarray} and the complex moment map,
\begin{eqnarray}\nu:\mbox{Rep}(\overline{\mathcal{Q}})\longrightarrow(\mathfrak{g}^*)^\CC.\label{complex}\end{eqnarray}

Each node $u$ of $\mathcal Q$ generates a component of the codomain of each moment map, which will be the dual of an appropriate Lie algebra (either $\mathfrak{u}(r_u)^*\cong\mathbb R^{r_u^2}$ or $\mathfrak{su}(r)^*\cong\mathbb R^{r^2-1}$ or their complexifications).  We can denote this component of $\mu$ (respectively, of $\nu$) by $\mu_u$ (respectively, by $\nu_u$).  

If $u\neq\star$, then we have$$\mu_u(x,y)\;=\;\displaystyle\left(\sum_{v\in V(\mathcal Q)}\sum_{e\in E_{v,u}(\mathcal Q)}x^e_{v,u}(x^e_{v,u})^*-(y^{-e}_{u,v})^*y^{-e}_{u,v}\right)-\left(\sum_{v\in V(\mathcal Q)}\sum_{e\in E_{u,v}(\mathcal Q)}(x^e_{u,v})^*x^e_{u,v}-y^{-e}_{v,u}(y^{-e}_{v,u})^*\right),$$where $^*$ denotes the conjugate transpose.  Note that this map does not actually take values in the Lie algebra, but rather in the image of the Lie algebra under multiplication by an imaginary factor.  For our purposes, this multiplication (which amounts of an isomorphism of complex varieties) is insignificant.  On the other hand, if $u=\star$, then we have$$\mu_\star(x,y,a,b)\;=\;\displaystyle\left(\sum_{v\in V(\mathcal Q)}\sum_{e\in E_{v,\star}(\mathcal Q)}x^e_{v,\star}(x^e_{v,\star})^*-(y^{-e}_{\star,v})^*y^{-e}_{\star,v}\right)_0+\sum_{\ell\in L(\mathcal Q)}[a^\ell,(a^\ell)^*]+[b^{-\ell},(b^{-\ell})^*],$$where the subscript $_0$ means that we have removed the trace.

We similarly have$$\nu_u(x,y)=\displaystyle\left(\sum_{v\in V(\mathcal Q)}\sum_{e\in E_{v,u}(\mathcal Q)}x^e_{v,u}y^{-e}_{u,v}\right)-\left(\sum_{v\in V(\mathcal Q)}\sum_{e\in E_{u,v}(\mathcal Q)}y^{-e}_{v,u}x^e_{u,v}\right)$$whenever $u\neq\star$, and$$\nu_\star(x,y,a,b)\;=\;\displaystyle\left(\sum_{v\in V(\mathcal Q)}\sum_{e\in E_{v,\star}(\mathcal Q)}x^e_{v,\star}y^{-e}_{\star,v}\right)_0+\sum_{\ell\in L(\mathcal Q)}[a^\ell,b^{-\ell}].$$

Let $\mathcal Z(\mathfrak g)$ denote the centre of the Lie algebra.  We can now define the following moduli spaces of representations of $\mathcal Q$ and $\overline{\mathcal Q}$, respectively:

\begin{definition}\label{DefnQV}\emph{If $\alpha\in\mathcal Z(\mathfrak g)$, then $$\mathcal P_{\mathcal Q}(\alpha)=\mbox{Rep}(\mathcal Q)\git_\alpha G:=\mu^{-1}(\alpha)/G$$and$$\mathcal X_{\mathcal Q}(\alpha)=\mbox{Rep}(\overline{\mathcal Q})\hkq_\alpha G:=(\mu^{-1}(\alpha)\cap\nu^{-1}(0))/G$$ are the \textbf{(ordinary) quiver variety} and the \textbf{Nakajima quiver variety}, respectively, associated to the labelled quiver $\mathcal Q$ at level $\alpha$.}\end{definition}

In this definition, the former variety is a symplectic (or Marsden-Weinstein \cite{MR402819}) quotient while the latter is a hyperk\"ahler (or Hitchin-Karlhede-Lindst\"om-Ro\v cek \cite{MR877637}) quotient. Here, the element $\alpha$ is regarded as a choice of symplectic (or K\"ahler) modulus.  Algebro-geometrically, the former is projective while the latter is quasiprojective (in fact, affine).  Both quotients can be realized as geometric invariant theory (GIT) quotients by $G^{\mathbb C}$, as per the Kempf-Ness Theorem \cite{MR555701}. In the definition of $\mathcal P_{\mathcal Q}(\alpha)$, all of the $y$ and $b$ inputs in the moment maps $\mu$ and $\nu$ are zet to $0$.  We denote elements of $\mathcal P_{\mathcal Q}(\alpha)$ by $[x,a]$ and elements of $\mathcal X_{\mathcal Q}(\alpha)$ by $[x,y,a,b]$.

By theorems of King \cite{king} and Nakajima \cite{HN}, these quotients are smooth whenever $\alpha$ is sufficiently generic.  It is also true that$$\dim_{\mathbb C}\mathcal P_{\mathcal Q}(\alpha)=\dim_{\mathbb C}\mbox{Rep}(\mathcal Q)-2\,\mbox{rank}_{\mathbb C}\,G$$and$$\dim_{\mathbb C}\mathcal X_{\mathcal Q}(\alpha)=2\dim_{\mathbb C}\mathcal P_{\mathcal Q}(\alpha).$$The dimension of the Nakajima quiver variety follows from the fact that the containment$$T^*\mathcal P_{\mathcal Q}(\alpha)\;\subset\;\mathcal X_{\mathcal Q}(\alpha)$$as is open dense  (and in some instances the complement is in fact empty).   The variety $P_{\mathcal Q}(\alpha)$ is the zero section of the bundle, which is where all of the $y$ maps are $0$.

Regarding the hyperk\"ahler geometry of $\mathcal X_{\mathcal Q}(\alpha)$, note that we are taking the quotient of three level sets (of $\mu$ and of the real and imaginary parts of $\nu$) inside a trivial hyperk\"ahler space.   If $i$ is the square root of $-1$ as per usual acting by multiplication on $T^*\mbox{Rep}(\mathcal Q)$, then $T^*\mbox{Rep}(\mathcal Q)$ is a quaternionic affine space with standard Hermitian norm $h$ quaternions $I,J,K$ given by
 \begin{eqnarray}
 I: (x,y,a,b) &\mapsto& (ix,iy,ia,ib),\label{i}\\
J: (x,y,a,b) &\mapsto& (-y^*,x^*, -b^*, a^*),\label{j}\\
K: (x,y,a,b) &\mapsto&  (-iy^*, ix^* ,  -ib^*, ia^*)\label{k}.
  \end{eqnarray}
  There are respective symplectic forms given by
  \begin{eqnarray}
  \omega_I(\cdot, \cdot)= h(I\cdot,\cdot),~~{\rm \hspace{.5 cm}}~~ \omega_J(\cdot, \cdot)= h(J\cdot,\cdot);~~~~{\rm \hspace{.5 cm}} \omega_K(\cdot, \cdot)= h(K\cdot,\cdot),
  \end{eqnarray}
  
The quotient $\mathcal X_{\mathcal Q}(\alpha)$ inherits $I,J,K$ and $h$, and the respective symplectic forms --- we repeat these names without confusion.  The quotient metric $h$ is now a nontrivial one, the Nakajima metric.  Its asymptotics are part of separate work (with H. Wei\ss); here, all we require is its existence.

\section{Generalized hyperpolygons}\label{gen1}

The moment maps in Section \ref{review} can be used to define (Nakajima) quiver varieties for any quiver with a \emph{central node} as defined in that section.  We now restrict our attention to particular shapes of quivers, leading to what we call moduli spaces of {\it generalized (hyper)polygons}.  Rather than doing this in one fell swoop, we will do this in stages, first by constructing partial flag varieties and then by interlacing their quivers to produce star-shaped and comet-shaped quivers.

\subsection{Flag varieties and comets}

A natural starting point for the whole subject of quiver varieties is the partial flag variety (over $\mathbb C$, for us).  We can think of a partial flag variety as an operation that associates to a string of positive integers $r_1<r_2<\dots<r_m$ a complex projective variety $\mathcal F_{r_1,\dots,r_m}$ that parametrizes, up to isomorphism, nested subspaces $V_1\subseteq\cdots\subseteq V_m$ whose dimensions are determined respectively by the string and where $V_m$ is fixed (i.e. its automorphisms are not permitted to act).  The string $1,2,\dots,r$ leads to the \emph{complete flag variety} of rank $r$.  On the other hand, the short string $1,r$ is the \emph{minimal flag variety} of rank $r$.  Computing the dimension of the flag variety is a standard exercise, from which we obtain$$f_{r_1,\dots,r_m}:=\dim_{\mathbb C}\mathcal F_{r_1,\dots,r_m}\;=\;\sum_{i=1}^{m-1}r_i(r_{i+1}-r_i).$$In the case of the complete flag variety, this dimension becomes $$f_{1,\dots,r}=\frac{r(r-1)}{2}.$$  In the case of the minimal flag variety, which only has one subspace, this dimension becomes $f_{1,r}=r-1$.  For economy, we will denote the string as a vector $\underline r=(r_1,\dots,r_m)$ and thus refer to $\mathcal F_{\underline r}$ and $f_{\underline r}$.  We will also put $[r]:=(1,2,\dots,r)$.

The flag variety $\mathcal F_{\underline r}$ is the (ordinary) quiver variety for an $A$-type quiver $\mathcal Q$ labelled by the string $\underline r$, with equioriented arrows pointing towards the largest number in the string, as in Figure \ref{atype} (a).   The node with the largest label, $r_m$, is declared to be the central node $\star$.  We may also consider the Nakajima quiver variety for $\overline{\mathcal Q}$ (see Figure \ref{atype} (b)), which in this case is identified precisely with $T^*\mathcal F_{\underline r}$ as a hypertoric variety.

\begin{figure}[!h]
\begin{center}
\includegraphics[width=0.55\linewidth]{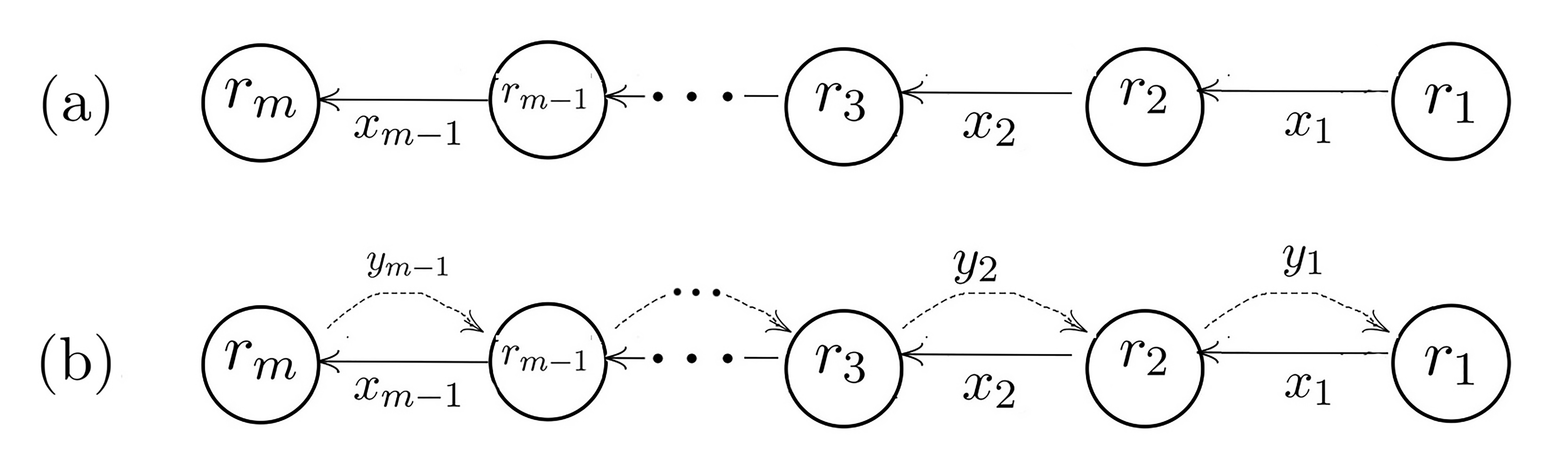}
\caption{(a) An $A$-type quiver $\mathcal Q$; (b) the corresponding doubled quiver $\overline{\mathcal Q}$.}\label{atype}
\end{center}
\end{figure}

To see how the variety $\mathcal F_{\underline r}$ arises as a quiver variety as in the sense of Definition \ref{DefnQV}, we fix the data of a matrix $\alpha\in\mathfrak{u}(r_1)$.   We construct a group $G$ associated to the quiver as in the previous section \emph{except} that we do not include the group determined by the $m$-th node, and so $G=U(1)^{m-1}$.  We denote by $x_{k-1}\in\mbox{Hom}(V_{k-1},V_k)\cong\mathbb C^{(k-1)k}$ a choice of representation for the arrow that points towards the $k$-th node; likewise, the arrow leaving that node is represented by some $x_k\in\mbox{Hom}(V_{k},V_{k+1})\cong\mathbb C^{k(k+1)}$, as in Figure \ref{atype} above. For each $k$ within $1<k<m$, the corresponding moment map for the $\mbox U(r_k)$-component of the action is$$\mu_i(x_1,\dots,x_{m-1})=x_{k-1}x_{k-1}^*-x_{k}^*x_{k},$$by the previous section. When $k=1$, this moment map is simply$$\mu_1(x_1,\dots,x_{m-1})=x_1^*x_1,$$where we omit a minus sign without loss of generality.  Then,$$\mathcal F_{\underline r}\;:=\;\mathcal P_{\mathcal Q}(\alpha,0,\dots,0)\;=\;\left(\mu_1^{-1}(\alpha)\cap\bigcap_{k=2}^{m-1}\mu_k(0)\right)/\prod_{k=1}^{m-1}\mbox U(r_k).$$Note that we take the choice of $\alpha$ to be understood when we write simply $\mathcal F_{\underline r}$.

Because of how we constructed $G$, there is a residual action of $\mbox{U}(r_m)$ on $\mathcal F_{\underline r}$ with moment map$$\mu_m(x_1,\dots,x_{m-1})=x_{m-1}x_{m-1}^*.$$The map $\mu_m:\mathcal F_{\underline r}\to\mathfrak{u}(r_m)$ embeds the flag variety as a coadjoint orbit for $\mbox U(r_m)$.  Likewise, the residual action of $\mbox U(r_m)^\mathbb{C}=\mbox{GL}(r_m,\mathbb C)$ on $T^*\mathcal F_{\underline r}$ has moment map$$\nu_m(x_1,\dots,x_{m-1};y_1,\dots,y_{m-1})=x_{m-1}y_{m-1}.$$The image of this map is a subvariety of the nilpotent cone in $\mathfrak{gl}(r_m,\mathbb C)$ and, for the complete flag, this map is the Springer resolution for the Lie algebra $\mathfrak{sl}(r,\mathbb C)$.

If we were to quotient of either $F_{\underline r}$ or $T^*\mathcal F_{\underline r}$ by the (complexified) residual action then we would obtain a zero-dimensional variety.  However, if take a number of $A$-type quivers that share a common maximal label $r$ and identify the nodes with those maximal labels, then we produce a new shape of quiver, the star-shaped quiver, an example of which appears in Figure \ref{atype2}.

\begin{figure}[h!]
\begin{center}
\includegraphics[width=0.8\linewidth]{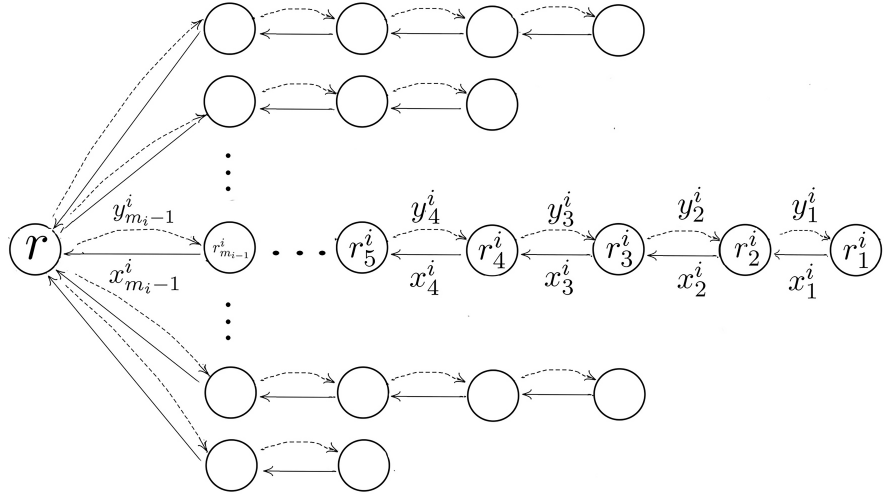}
\caption{Several $A$-type quivers sharing a common maximal label $r$ have been joined together to form a {\it  star-shaped quiver} $\overline{\mathcal Q}$.}\label{atype2}
\end{center}
\end{figure}
\pagebreak
If there are sufficiently-many $A$-type quivers glued in this way and we quotient by the full action (i.e. by the group $G$ that includes the automorphisms of the central node), the dimension of the resulting quiver variety will be nonzero.  We refer to the $A$-type quivers as the \emph{arms} of the star-shaped quiver.   
In Figure \ref{atype2} above, the $i$-th arm of the star shaped quiver has been given labels. To the star-shaped quiver, we also add $g$ loops to the central node of $\mathcal Q$ (hence, $2g$ loops to the central node of $\overline{\mathcal Q}$).  This produces the so-called \emph{comet quiver}, an example of which is shown in Figure \ref{comet} below.

  \begin{figure}[h!]
\begin{center}
\includegraphics[width=0.8\linewidth]{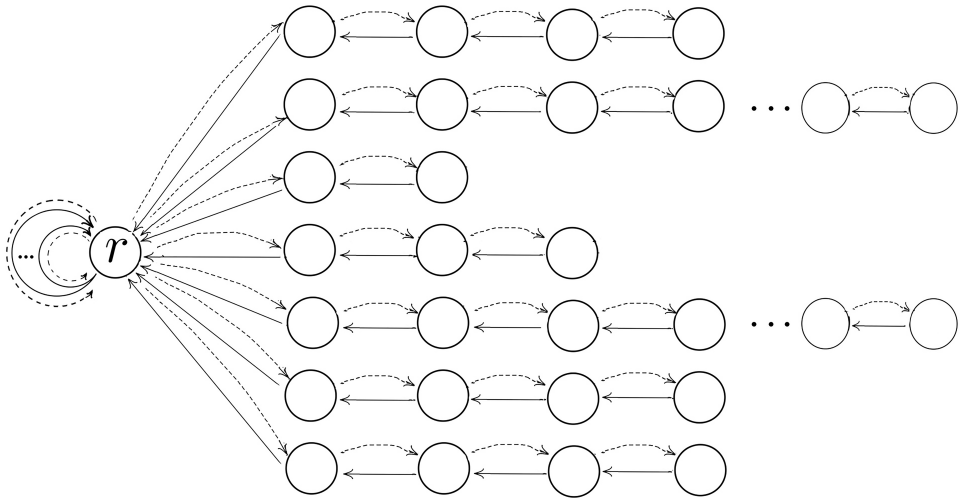}
\caption{A comet-shaped quiver.} \label{comet}
\end{center}
\end{figure}

    As per the conventions of the previous section, a representation of the $j$-th loop of $\mathcal Q$ is a choice of matrix $a_j\in\mathfrak{sl}(r,\mathbb C)$.  Suppose that there are $n$ arms.  Consider the $i$-th arm and let its string be given by the vector $\underline{r^i}=(r^i_1,r^i_2,\dots,r^i_{m_i})$, where $r^i_{m_i}=r$, as in Figure \ref{atype2}.  When each and every arm is the complete flag, we refer to the quiver as the \emph{complete comet}, as in Figure \ref{comet1} (a).  When each and every arm is the minimal flag, we refer to the quiver as the \emph{minimal comet},  as in Figure \ref{comet1} (b).

\begin{figure}[!h]
\begin{center}
\includegraphics[width=0.85\linewidth]{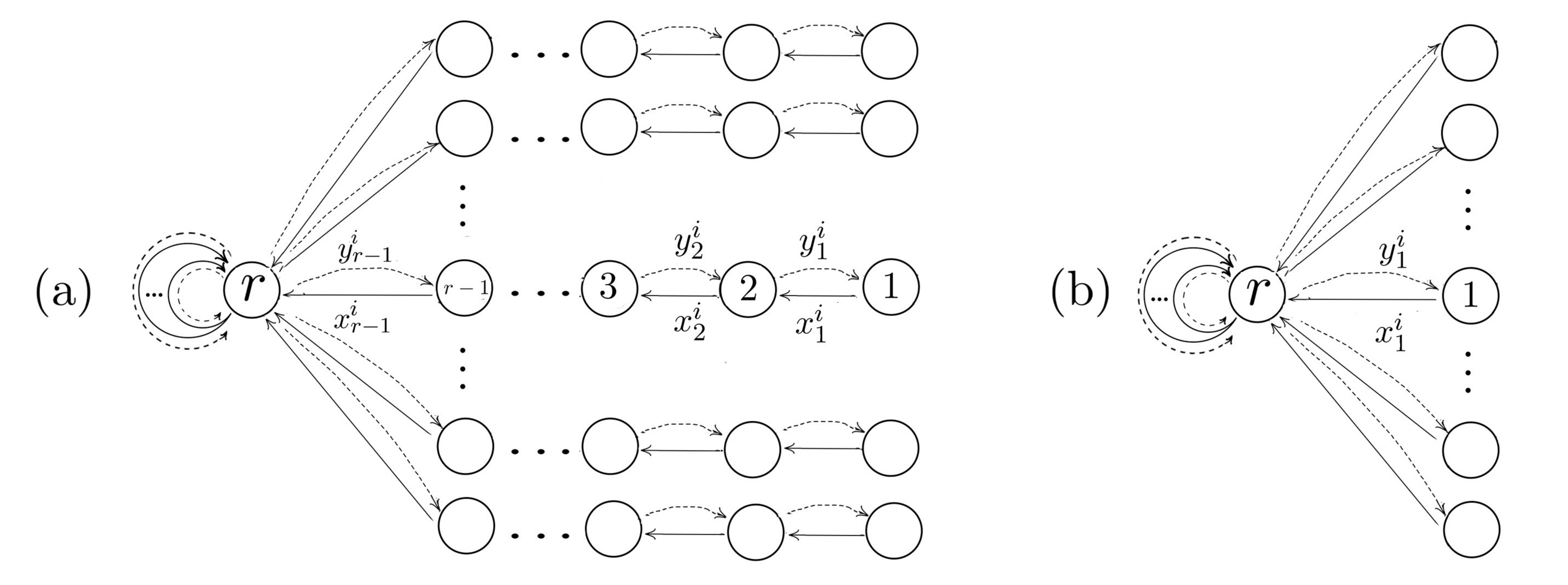}
\caption{(a) A complete comet; (b) a minimal comet.} \label{comet1}
\end{center}
\end{figure}

   Let the ordinary quiver variety associated to the comet quiver be denoted $\mathcal P^g_{\underline{r^1},\dots,\underline{r^n}}(\underline\alpha)$.  This quiver variety arises, as a symplectic quotient, from the choice $\underline\alpha$ of $n$ matrices $\alpha_i\in\mathfrak{u}(r^i_1)$.  Considering the residual action on the central node by $\mbox{SU}(r)$, the associated moment map is now$$\mu_\star:\left(\mathcal F_{\underline{r^1}}(\alpha_1)\times\cdots\times\mathcal F_{\underline{r^n}}(\alpha_n)\right)\times\mathfrak{sl}(r,\mathbb C)^g\to\mathfrak{su}(r)$$given by$$\mu_\star(x,a)\;=\;\sum_{i=1}^n(x^i_{m_i-1}(x^i_{m_i-1})^*)_0+\sum_{j=1}^g[a_j,a_j^*],$$where the subscript $0$ is again an instruction to remove the trace (and we have identified Lie algebras with their duals).  Then, we have$$\mathcal P^g_{\underline{r^1},\dots,\underline{r^n}}(\underline\alpha)\;:=\;\mu_\star^{-1}(0)/\mbox{SU}(r).$$

\begin{proposition}\label{PropDimP}  The complex dimension of $\mathcal P^g_{\underline{r^1},\dots,\underline{r^n}}(\underline\alpha)$ is $\sum_{i=1}^nf_{\underline{r^i}}+(g-1)(r^2-1)$.\end{proposition}

\begin{proof}  The formula follows immediately from the fact that the quotient variety $\mathcal P^g_{\underline{r^1},\dots,\underline{r^n}}(\underline\alpha)$ is equivalent, by the Kempf-Ness Theorem, to the geometric quotient of the Cartesian product$$\left(\mathcal F_{\underline{r^1}}(\alpha_1)\times\cdots\times\mathcal F_{\underline{r^n}(\alpha_n)}\right)\times\mathfrak{sl}(r,\mathbb C)^g$$by the diagonal $\mbox{SL}(r,\mathbb C)$ action.\end{proof}

Note that the dimension is independent of the choices of the $\alpha_i$.

\begin{corollary} The complex dimension of $\mathcal P^g_{[r],\dots,[r]}(\underline\alpha)$ (i.e. where the quiver is the complete comet) is$$\frac{nr(r-1)}{2}+(g-1)(r^2-1).$$The complex dimension of $\mathcal P^g_{(1,r),\dots,(1,r)}(\underline\alpha)$ (i.e. where the quiver is the minimal comet) is$$n(r-1)+(g-1)(r^2-1).$$\end{corollary}

For the Nakajima quiver variety of the comet, the corresponding moment maps are$$\mu_\star:\left(T^*\mathcal F_{\underline{r^1}}(\alpha_1)\times\cdots\times T^*\mathcal F_{\underline{r^n}}(\alpha_n)\right)\times T^*\mathfrak{sl}(r,\mathbb C)^g\to\mathfrak{su}(r)$$$$\nu_\star:\left(T^*\mathcal F_{\underline{r^1}}(\alpha_1)\times\cdots\times T^*\mathcal F_{\underline{r^n}}(\alpha_n)\right)\times T^*\mathfrak{sl}(r,\mathbb C)^g\to\mathfrak{sl}(r,\mathbb C)$$given respectively by$$\mu_\star(x,y,a,b)\;=\;\sum_{i=1}^n(x^i_{m_i-1}(x^i_{m_i-1})^*-(y^i_{m_i-1})^*y^i_{m_i-1})_0+\sum_{j=1}^g[a_j,a_j^*]+[b_j,b_j^*],$$$$\nu_\star(x,y,a,b)\;=\;\sum_{i=1}^n(x^i_{m_i-1}y^i_{m_i-1})_0+\sum_{j=1}^g[a_j,b_j].$$Then, we have$$\mathcal X^g_{\underline{r^1},\dots,\underline{r^n}}(\underline\alpha)\;:=\;(\mu_\star^{-1}(0)\cap\nu_\star^{-1}(0))/\mbox{SU}(r).$$

We are particularly interested in the quotient $\mathcal X^g_{\underline{r^1},\dots,\underline{r^n}}(\underline\alpha)$, which has dimension twice that in Proposition \eqref{PropDimP}.  Note that the containment of $T^*\mathcal P^g_{\underline{r^1},\dots,\underline{r^n}}(\underline\alpha)$ is now \emph{strict} and so the projection $\mathcal X^g_{\underline{r^1},\dots,\underline{r^n}}(\underline\alpha)\to\mathcal P^g_{\underline{r^1},\dots,\underline{r^n}}(\underline\alpha)$ is only rationally defined.

The relationship of $\mathcal P^g_{\underline{r^1},\dots,\underline{r^n}}(\underline\alpha)$ and $\mathcal X^g_{\underline{r^1},\dots,\underline{r^n}}(\underline\alpha)$ to polygons and hyperpolygons, respectively, is the subject of Section \eqref{SectGeomHyp}.

\subsection{Geometry of generalized (hyper)polygons}\label{SectGeomHyp}

When $r^i_1=1$ for each arm of the comet, each matrix $\alpha_i$ is simply a real number and we will always take $\alpha_i\in\mathbb R_{>0}$.  In this situation, we refer to isomorphism classes of representations of the comet --- that is, classes $[x,a]\in\mathcal P^g_{\underline{r^1},\dots,\underline{r^n}}(\underline\alpha)$ --- as \emph{generalized $n$-gons of rank $r$, genus $g$, and length $\underline\alpha$}.  These objects generalize isomorphism classes of ordinary $n$-gons (closed, with fixed side lengths and barycentre) in the Euclidean space $\mathbb R^{3}$, which are elements of $\mathcal P^0_{(1,2),\dots,(1,2)}(\underline\alpha)$.  As such, we refer to $\mathcal P^g_{\underline{r^1},\dots,\underline{r^n}}(\underline\alpha)$ as the moduli space of generalized $n$-gons of rank $r$, genus $g$, and length $\underline\alpha$.

A representation in $\mathcal P^g_{\underline{r^1},\dots,\underline{r^n}}(\underline\alpha)$ determines a closed polygon in $\mathbb R^{r^2-1}\cong\mathfrak{su}(r)$ in the following way: the sides are given by the matrices $(x_{m_i-1}^i(x_{m_i-1}^i)^*)_0$, which are determined respectively by the representation of the closest arrow along the $i$-th arm to the central node, as well as the matrices $[a_j,a_j^*]$ determined by the loops.  From the symplectic point of view, the moment map condition $\mu_\star=0$ on$$\left(\mathcal F_{\underline{r^1}(\alpha_1)}\times\cdots\times\mathcal F_{\underline{r^n}(\alpha_n)}\right)\times\mathfrak{sl}(r,\mathbb C)^g$$is equivalent to asking that these matrices, when viewed as vectors in $\mathbb R^{r^2-1}$, make a closed figure with $n+g$ sides.  If we move outward along any arm, each node (save for the terminal node with label $r^i_1=1$) produces the moment map condition$$(x^i_{k})^*x^i_{k}=x_{k-1}^i(x_{k-1}^i)^*$$in $\mathbb R^{(r^i_k)^2-1}\cong\mathfrak{su}(r^i_k)$ from the associated partial flag variety $\mathcal F_{\underline{r^i}}(\alpha_i)$.  The terminal node yields the single condition$$(x^i_1)^*(x^i_1)=|x^i_1|^2=\alpha_i.$$  In combination with the other moment map conditions, the condition $|x^i_1|^2=\alpha_i^2$ fixes the norms of the matrices $x^i_k(x^i_k)^*$ and $(x^i_k)^*x^i_k$ for all $k$.  In particular, each side $(x_{m_i-1}^i(x_{m_i-1}^i)^*)_0$ has Euclidean length given by $\frac{1}{\sqrt r}\alpha_i$ and the length of $\sum_{j=1}^g[a_j,a_j^*]$ is given by $\frac{1}{\sqrt r}\sum\alpha_i$.

Classes in $\mathcal X^g_{\underline{r^1},\dots,\underline{r^n}}(\underline\alpha)$ have a similar geometric interpretation as \emph{pairs} of polygons, one with $2n+2g$ sides in $\mathfrak{su}(r)$ and the other with $n+g$ sides in $\mathfrak{sl}(r,\mathbb C)$.  The data of a representation $[x,y,a,b]$ satisfying the moment map equations yield a closed figure in $\mathfrak{su}(r)$ with sides of the form $(x_{m_i-1}^i(x_{m_i-1}^i)^*)_0$, $-((y_{m_i-1}^i)^*x_{m_i-1}^i)_0$, $[a_j,a_j^*]$, and $[b_j,b_j^*]$, with $i$ ranging from $1$ to $n$ and $j$ ranging from $1$ to $g$.  At the same time, we have a closed figure in $\mathfrak{sl}(r,\mathbb C)$ with sides of the form $(x_{m_i-1}y_{m_i-1})_0$ and $[a_j,b_j]$.  We refer to pairs of polygons of this form as \emph{hyperpolygons}.  Specifically, the pair here is a \emph{hyper-$n$-gon of rank $r$, genus $g$, and length $\underline\alpha$} and $\mathcal X^g_{\underline{r^1},\dots,\underline{r^n}}(\underline\alpha)$ is their moduli space.   Note that, for the polygon in $\mathfrak{su}(r)$, the length of the side $(x_{m_i-1}^i(x_{m_i-1}^i)^*)_0$ needs not be $\frac{1}{\sqrt r}\alpha_i$.  Rather, it is the length of the \emph{difference} $(x_{m_i-1}^i(x_{m_i-1}^i)^*)-(y_{m_i-1}^i)^*x_{m_i-1}^i)_0$ that must be $\frac{1}{\sqrt r}\alpha_i$.  This feature makes plain the noncompactness of $\mathcal X^g_{\underline{r^1},\dots,\underline{r^n}}(\underline\alpha)$ (as opposed to $\mathcal P^g_{\underline{r^1},\dots,\underline{r^n}}(\underline\alpha)$, which is necessarily compact).

We refer to the polygon in $\mathfrak{su}(r)$ as the \emph{bundle polygon} and the one in $\mathfrak{sl}(r,\mathbb C)$ as the \emph{Higgs polygon}.  The discussion in Section \ref{SecAnalogy} will make clear the reason for this nomenclature.

\subsection{$\mbox{U}(1)$-action on hyperpolygons}

The quasiprojective variety $\mathcal X^g_{\underline{r^1},\dots,\underline{r^n}}(\underline\alpha)$ comes with an action of $\mbox{U}(1)$ defined as so: if $[x,y,a,b]\in\mathcal X^g_{\underline{r^1},\dots,\underline{r^n}}(\underline\alpha)$ is a generalized hyperpolygon, then we have$$[x,y,a,b]\stackrel{\theta}{\longrightarrow}[x,\exp(i\theta)y,a,\exp(i\theta)b].$$It is easy enough to check that this action is Hamiltonian with regards to the $\omega_I$ symplectric form.  Had we defined $\mathcal X^g_{\underline{r^1},\dots,\underline{r^n}}(\underline\alpha)$ as a GIT quotient, this action can be promoted to an algebraic action by $\mathbb C^*$, preserving the $I$ complex structure.

This action was used in Section 3 of \cite{steve} to compute, via Morse-Bott localization, the Betti numbers of the rational cohomology ring in the case of the minimal comet with no loops.  The calculations would be similar in the case of an arbitrary comet, although as our interests are more geometrical rather than topological here, we leave this aside.  Rather, we focus on the following observation.  While every Nakajima quiver variety enjoys such an action, it is interesting to note its interpretation in terms of the geometry of hyperpolygons.  Note that the expressions $((y_{m_i-1}^i)^*y_{m_i-1}^i)_0$ and $[b_j,b_j^*]$ are invariant under the action, and so the bundle polygon is invariant geometrically.  On the other hand, the expressions $(x_{m_i-1}y_{m_i-1})_0$ and $[a_j,b_j]$ are not generally invariant under the action.  In particular, the moduli space of generalized polygons $\mathcal P^g_{\underline{r^1},\dots,\underline{r^n}}(\underline\alpha)$, which sits inside $\mathcal X^g_{\underline{r^1},\dots,\underline{r^n}}(\underline\alpha)$ as the zero locus of $T^*\mathcal P^g_{\underline{r^1},\dots,\underline{r^n}}(\underline\alpha)$, is fixed under the action.

As such, in any hyperpolygon $[x,y,a,b]$ the bundle polygon is static while the Higgs polygon is rotated through $\theta$ within the ambient Lie algebra.  This feature of the $\mbox{U}(1)$-action resembles that of the action on the moduli space of Higgs bundles on a compact Riemann surface \cite{N1}, which has a Hamiltonian $\mbox{U}(1)$-action (respectively, an algebraic $\mathbb C^*$-action).   This leads to a localization for the Higgs bundle cohomology, as originally done in rank $2$ in \cite{N1} (cf. \cite{MR3884746} for a survey in arbitary rank).  On the Higgs bundle moduli space, the action leaves vector bundles invariant while rotating Higgs fields --- and in particular, the moduli space of stable bundles is invariant.   This similarity between hyperpolygons and Higgs bundles motivates the analogy of the next section.

\section{Relationship to Higgs bundles}\label{SecAnalogy}

\subsection{Analogy with Hitchin equations}

We take $r^i=1$, $r^i_m=r$ from now on, to maintain the hypotheses of the preceding section.  The construction above realizes the Nakajima quiver variety $\mathcal X^g_{\underline{r^1},\dots,\underline{r^n}}(\underline\alpha)$ as the set of solutions to a set of equations in spaces of matrices, divided by the complexified adjoint action of the group $G=\mbox{SU}(r)\times U(1)^n$. We isolote these equations here:

\begin{definition} For each $\underline\alpha\in\mathbb R^n_{>0}$, the \textbf{hyperpolygon equations} are:
 \begin{eqnarray}
& \emph{(i)} & \sum_{i=1}^n(x^i_{m_i-1}(x^i_{m_i-1})^*-(y^i_{m_i-1})^*y^i_{m_i-1})_0+\sum_{j=1}^g[a_j,a_j^*]+[b_j,b_j^*]  =  0\nonumber\\
& \emph{(ii)} & x_{k-1}^i(x_{k-1}^i)^*+y^i_k(y^i_k)^*-(x^i_{k})^*x^i_{k}-(y^i_{k-1})^*y_{k-1}^i =  0,\;k=2,\dots,m_i-2,\;i=1,\dots,n\nonumber\\
& \emph{(iii)} & |x^i_1|^2-|y^i_1|^2  =  \alpha_i,\;i=1,\dots,n\nonumber\\
& \emph{(I)} & \sum_{i=1}^n(x^i_{m_i-1}y^i_{m_i-1})_0+\sum_{j=1}^g[a_j,b_j]  =  0\nonumber\\
& \emph{(II)} & x_{k-1}^iy_{k-1}^i-y_k^ix_k^i=0,\;k=2,\dots,m_i-2,\;i=1,\dots,n\nonumber\\
& \emph{(III)} & y^i_1x^i_1  =  0,\;i=1,\dots,n\nonumber
\end{eqnarray}\label{DefnHEqns}
\end{definition}

We single out these equations for the reason that they are discrete analogues of the \emph{Hitchin equations} \cite{N1}, which are equations on a smooth Hermitian bundle $E$ over a Riemann surface $X$.  In particular, equation (i) is the analogue of the first Hitchin equation$$F(A)+\phi\wedge\phi^*=0$$in which the failure of a unitary connection $A$ on $E$ to be flat is expressed in terms of a \emph{Higgs field} $\phi:E\to E\otimes K$, where $\phi$ is linear in functions (i.e. $\phi(fs)=f\phi(s)$ for any section $s\in\Gamma(E)$ and any function $f\in\mathcal C^\infty(X)$) and $K$ is the canonical line bundle of $X$ (i.e. the bundle of holomorphic one-forms).  Here, we may identify the term $\sum_{i=1}^n(x^i_{m_i-1}(x^i_{m_i-1})^*)_0+\sum_{j=1}^g[a_j,a_j^*]$ with $F(A)$ and the term $-((y^i_{m_i-1})^*y^i_{m_i-1})_0+\sum_{j=1}^g[b_j,b_j^*]$ with $\phi\wedge\phi^*$.  In particular, the failure of the connection to be flat is paralleled in the failure of the (ordinary) polygon to close.  The way in which the Higgs field ``flattens'' the connection (we have $F(A+\phi+\phi^*)=0$) is analogous to the way in which the $y$ and $b$ data close the figure in $\mathfrak{su}(r)$.  Likewise, equation (I) is the analogue of the second Hitchin equation$$d_A''\phi=0$$that makes $\phi$ holomorphic with respect to the holomorphic structure on $E$ induced by $A$.  Equation (I) can be thought of as ``holomorphicity at infinity'' for an associated Higgs bundle, which we describe now.

\subsection{Associated meromorphic Higgs bundle and tame character varieties}

In the $g=0$ (i.e. star-shaped) case with sufficiently generic $\underline\alpha$, one can identify elements $[x,y]\in\mathcal X^0_{\underline{r^1},\dots,\underline{r^n}}(\underline\alpha)$ with meromorphic Higgs bundles of rank $r$ on an $n$-punctured complex projective line, where $\underline\beta$ is an appropriate choice of parabolic weight vector along the divisor of punctures.  This was accomplished in Section 4 of \cite{steve} in the case that each arm has the minimal flag type.   The construction of an associated Higgs bundle from \cite{steve} is actually independent of the flag type.  To see this, let $E$ stand for $\mathbb P^1\times\mathbb C^r$ with the trivial holomorphic structure and let $D=\sum_{i=1}^nz_i$ be the divisor, such that the $z_i$ are pairwise distinct and none are $\infty$.  Then, define$$\phi_{[x,y]}(z)\;=\;\sum_{i=1}^n\frac{(x^i_{m_i-1}y^i_{m_i-1})_0}{z-z_i}dz.$$The pair $(E,\phi)$ defines a Higgs structure on the trivial rank-$r$ bundle on $\mathbb P^1$ that is memorphic along $D$, associated to the generalized genus-$0$ hyperpolygon $[x,y]$.  Equation (I), which is a condition on the $z_i$-residues of $\phi_{[x,y]}$, ensures that $\phi_{[x,y]}$ is holomorphic at $\infty\in\mathbb P^1$.  Specifically for $g=0,r=2$, it was shown in \cite{Ale} that the map$$[x,y]\mapsto[\phi_{[x,y]}]$$is an embedding of moduli spaces for $\underline\alpha$ sufficiently generic, where $[\phi_{[x,y]}]$ is the isomorphism class of $\phi_{[x,y]}$ under the conjugation action of holomorphic automorphisms of $E$.  The corresponding moduli space of the Higgs fields $\phi_{[x,y]}$ is defined using slope stability with respect to a parabolic weight vector $\underline\beta$.  The weights $\underline\beta$ can be used to turn $(E,\phi_{[x,y]})$ into a \emph{strictly} parabolic Higgs bundle and are determined, albeit non-uniquely, from $\underline\alpha$.  This can be done for any $r$ but, for our purposes, we will not need to compute $\beta$ explicitly.  Note also that the residues are nilpotent of order equal to the length of the corresponding arm.  For example, if the $i$-th arm is complete, then $(x^i_{m_i-1}y^i_{m_i-1})_0^r=0$ by the moment map conditions.  For a minimal arm, we have $(x^i_{m_i-1}y^i_{m_i-1})_0^2=0$.

The above construction of a Higgs bundle from a hyperpolygon can be extended to higher genus.  The easiest way to accomplish this is to consider $(X,\mathcal C)$, where $X$ is either the standard complex or hyperbolic plane and $\mathcal C$ is a regular tiling of $X$.  We then define a matrix-valued one-form on $X$ from $[x,y,a,b]\in\mathcal X^g_{\underline{r^1},\dots,\underline{r^n}}(\underline\alpha)$ as follows.  First, if $g=0$, then $X$ is the standard plane with the fundamental cell equal to the whole plane; if $g=1$, then $X$ is the standard plane with $\mathcal C$ a rhombus tiling that is regular with regards to the standard Euclidean metric; or if $g>1$, then $X$ is the hyperbolic plane with $\mathcal C$ a tiling by $2g$-gons with a unit cell that is regular with regards to the Poincar\'e metric.   Then, we choose $n$ distinct points $z_i$ in the fundamental cell and construct$$\phi_{[x,y,a,b]}(z)\;=\;\sum_{i=1}^n\frac{(x^i_{m_i-1}y^i_{m_i-1})_0}{z-\iota_z(z_i)}dz,$$where $\iota_z:\mathbb C\to\mathbb C$ translates $z_i$ to the copy of the fundamental cell to which $z$ belongs (which can be accomplished by choosing an appropriate element of a Fuchsian group $\Gamma$).   When $X=\CC$ and the cell is all of $X$, we obtain a meromorphic Higgs field for the trivial vector bundle $E=\mathbb P^1\times\mathbb C^r$ by compactifying the plane.  (Here, $\iota_z$ is the identity.)  When $X=\CC$ and we have the lattice determined by the rhombus tiling, then $\phi_{[x,y,a,b]}$ descends to a well-defined object on the compactification of the quotient by the lattice, yielding a meromorphic Higgs field for the trivial rank-$r$ bundle on an elliptic curve $C$, with complex structure determined by the modulus of the cell and punctures along $D=\sum z_i$.  In the hyperbolic case, $\phi_{[x,y,a,b]}$ descends to a meromorphic Higgs field for the trivial rank-$r$ bundle on genus-$g\geq2$ Riemann surface $C$ punctured along $D=\sum z_i$.  The condition$$\sum_{i=1}^n(x^i_{m_i-1}y^i_{m_i-1})_0=\sum_{j=1}^g[b_j,a_j]$$intertwines the residues of $\phi_{[x,y,a,b]}$ at the punctures with the fundamental group of the surface.  The construction of $\phi_{[x,y,a,b]}$ as a Higgs field for the trivial bundle on a genus-$g$ surface makes clear the descriptor ``of genus $g$'' attached to our generalized hyperpolygons.

Note that two different representatives of the class $[x,y,a,b]\in\mathcal X^g_{\underline{r^1},\dots,\underline{r^n}}(\underline\alpha)$ differ by an element $g\in\mbox{GL}(r,\mathbb C)$.  At the same time, $g$ transforms $\phi_{[x,y,a,b]}$ by$$\phi_{[x,y,a,b]}\mapsto g^{-1}\phi_{[x,y,a,b]}g,$$which is precisely the notion of equivalence for Higgs fields for the trivial rank-$r$ bundle.  As in the $g=0$ case, we can also prove that genericity of $\underline\alpha$ corresponds with $\underline\beta$-stability for a parabolic structure on $E$ induced by the flags and $\underline\alpha$.  As the weights are unimportant to us,  the meromorphic Higgs bundle $(E,\phi_{[x,y,a,b]})$ can be transformed by clearing denominators into a so-called ``twisted'' Higgs bundle or ``Hitchin pair'' $(E,\widetilde\phi_{[x,y,a,b]})$ (i.e. where the Higgs field is of the form $\phi:E\to E\otimes K\otimes L$ for some holomorphic line bundle $L$).  Such objects enjoy the same notion of equivalence but are subject to the usual Mumford-Hitchin slope stability (as opposed to parabolic stability).  Furthermore, because the underlying bundle $E$ is trivial, the resulting twisted Higgs bundle is automatically semistable.  This can be packaged into the following result:

\begin{theorem}\label{ThmEmbed} Let the map $[x,y,a,b]\mapsto[\phi_{[x,y,a,b]}]$ and the quotient Riemann surface $C$ be as above.  The map embeds $X^g_{\underline{r^1},\dots,\underline{r^n}}(\underline\alpha)$ into the locus of the moduli space of rank-$r$, degree-$0$, slope-semistable twisted Higgs bundles on $C$ consisting of pairs $(E,\widetilde\phi_{[x,y,a,b]})$ in which $E$ is holomorphically trivial.\end{theorem}

\subsection{Comparison in the $\widetilde{\mbox{D}}_4$ case}

For an illuminating example of the relationship between hyperpolygons, Higgs bundles, and their moduli spaces, we consider the case where $g=0,n=4$ and each flag is of the form $\underline r^i=(1,2)=[2]$.  Then, $\mathcal X^0_{[2],[2],[2],[2]}(\underline\alpha)$ is the Nakajima quiver variety for the affine Dynkin diagram $\widetilde{\mbox{D}}_4$, which is the star-shaped quiver with $4$ outer nodes.  Furthermore, its labelling is such that the complex dimension of the quotient is $2$.  Hence, by the aforementioned McKay-Kronheimer-Nakajima correspondence, $\mathcal X^0_{[2],[2],[2],[2]}(\underline\alpha)$ is a noncompact, complete hyperk\"ahler $4$-manifold admitting the structure of an ALE gravitational instanton.  Furthermore, by virtue of its being a complex variety, it is also a noncompact K3 surface.  In this case, the associated Higgs bundles are parabolic of rank $2$ on $\mathbb P^1$ with $4$ tame punctures. Their stable moduli space is well known to be a noncompact K3 surface with a hyperk\"ahler metric --- Hitchin's $L^2$-metric --- which is also complete and  has recently been shown to be ALG \cite{laura2020}.  Unlike the Nakajima metric, the Hitchin one arises from an \emph{infinite-dimensional} hyperk\"ahler quotient.  As the two metrics are both already complete, the embedding of moduli spaces cannot be an embedding of hyperk\"ahler varieties (in fact, they are compatible only in the $I$ complex structures).  Also, in the Higgs moduli space, there are Higgs bundles with underlying bundle $E=\mathcal O(1)\oplus\mathcal O(-1)\cong K^{-1/2}\oplus K^{1/2}$, where $K^{1/2}$ is the unique spin structure (i.e. holomorphic square root of $K$) on the line.  These Higgs bundles are the complement of the embedding.  In other words, the difference between $\mathcal X^0_{[2],[2],[2],[2]}(\underline\alpha)$ and the associated parabolic Higgs bundle moduli space is precisely what is known as the \emph{Hitchin section} of the Hitchin fibration.  (See also Section 4.1 ``H3 surfaces'' in \cite{MR3931781} for further discussion of these examples.)

Hence, while there is an embedding of $\mathcal X^0_{[2],[2],[2],[2]}(\underline\alpha)$ into a Higgs moduli space, the metrics are incompatible and one should view the actual geometric relationship in a different way. A more universal point of view is that $\mathcal X^g_{\underline{r^1},\dots,\underline{r^n}}(\underline\alpha)$ is a \emph{linearization} of an appropriately-defined character variety on a punctured, tame genus-$g$ curve, which of course is diffeomorphic to a corresponding Higgs bundle moduli space by nonabelian Hodge theory.  This point of view is made plain in equation (I), which is the linearization of the usual character variety definition on a punctured surface.  The character variety itself should be thought of as a \emph{multiplicative quiver variety}, defined using group-valued moment maps and quasi-Hamiltonian reduction (e.g. \cite{MR2352135,MR3447108}).  The geometry of the multiplicative variety converges less rapidly to a Euclidean geometry at infinity than the corresponding ``additive'' varieties that we are considering.  This is the origin of the ALE / ALG difference in the $\widetilde{\mbox{D}}_4$ example.  (One may also wish to compare $\mathcal X^0_{[2],[2],[2],[2],[2]}(\underline\alpha)$ to the $5$-punctured case in \cite{MR3931793}.)  The topic of the asymptotic geometry of the Nakajima metric on hyperpolygon spaces and how it compares to the Hitchin metric is the topic of forthcoming work of the first named author and H. Wei\ss.

\section{The integrable system}\label{gen3}

The relationship of hyperpolygons to Higgs bundles raises the question of whether $\mathcal X^g_{\underline{r^1},\dots,\underline{r^n}}(\underline\alpha)$ possesses an integrable system, in the spirit of the Higgs bundle moduli space \cite{N2}.  It is generally expected that Nakajima quiver varieties ought to be algebraically completely integrable Hamiltonian systems with a Hitchin-like fibration (cf. the commentary at the end of \cite{HN}).  For the minimal-flag moduli space $\mathcal X^0_{[1,r],\dots,[1,r]}(\underline\alpha)$ for $r\leq3$ and arbitrary $n$, this was proven explicitly in Section 4 of \cite{steve} by embedding these spaces into an associated tame parabolic Higgs bundle moduli space on the punctured sphere.  For $r=3$, the minimality of the flags means that image of the embedding $[x,y]\mapsto[\phi_{[x,y]}]$ lies in a non-generic locus of the associated Hitchin system where the derivative of the Hitchin map drops in rank and thus extra analysis is required to demonstrate that the image constitutes a sub-integrable system (cf. \cite{hitchin2019sub} for further inquiries along this theme in the context of the Hitchin system).

Owing to the map $[x,y,a,b]\mapsto[\phi_{[x,y,a,b]}]$ one way to proceed, as in \cite{steve}, is as follows: to equip $\mathcal X^g_{\underline{r^1},\dots,\underline{r^n}}(\underline\alpha)$ with appropriate Hamiltonians, we take the real and imaginary parts of the components of the characteristic polynomial of the one-form-valued endomorphism $\phi_{[x,y,a,b]}$.  This equips $\mathcal X^g_{\underline{r^1},\dots,\underline{r^n}}(\underline\alpha)$ with the Hitchin Hamiltonians coming from the $I$ complex structure on the corresponding parabolic Higgs bundle moduli space (cf. \cite{MR2746468} for the non-strict parabolic case and \cite{MR3815160} for the strict case, for instance). Equivalently, we can again clear the denominators in $\phi_{[x,y,a,b]}$ and produce an associated matrix-valued polynomial $\widetilde\phi_{[x,y,a,b]}$, whose characteristic coefficients we then extract. This equips $\mathcal X^g_{\underline{r^1},\dots,\underline{r^n}}(\underline\alpha)$ with the Beauville-Markman Hamiltonians \cite{MR1300764} from the moduli space of twisted Higgs bundles.  In the case of the complete comet, the integrability is assured because the embedding of moduli spaces intersects \emph{every} Hitchin fibre in parabolic Higgs moduli space and so the Hamiltonians remain globally independent.  Here, the codomain of the Hamiltonians is$$B=\bigoplus_{i=2}^rH^0(X,K^{\otimes i}(D^{i-1})),$$ where $X$ and $D$ are the Riemann surface and divisor constructed in the previous section.  The affine space $B$ is in fact the base of the corresponding Hitchin fibration for the strictly parabolic Higgs bundle moduli space with complete flags at the punctures.   In the event of a flag that is incomplete, then the residues will not be generic and only certain non-generic Hitchin fibres (corresponding to singular spectral curves) will intersect the image of the embedding.  In this case, one has to establish functional independence by some technique, such as the algebraic disingularization technique in Section 4 of \cite{steve} used for minimal flags.  In this case of the minimal comet, we have$$B=\bigoplus_{i=2}^rH^0(X,K^{\otimes i}(D)).$$

In some sense, it is more satisfying to have a description of integrability of $\mathcal X^g_{\underline{r^1},\dots,\underline{r^n}}(\underline\alpha)$ that \emph{does not} rely upon an embedding into a Hitchin moduli space, as the embedding necessitates a choice of marked Riemann surface.  Put differently, there ought to be an \emph{intrinsic} set of Hamiltonians on $\mathcal X^g_{\underline{r^1},\dots,\underline{r^n}}(\underline\alpha)$.   We accomplish this in the case of the complete and minimal comets by appealing to the Gelfand-Tsetlin integrable system on each $T^*\mathcal F_{\underline{r^i}}$.

\begin{theorem}\label{ThmIntegrable} When each arm of the comet quiver is either complete or minimal, the moduli space $\mathcal X^g_{\underline{r^1},\dots,\underline{r^n}}(\underline\alpha)$ is an algebraically completely integrable Hamiltonian system of Gelfand-Tsetlin type with Hamiltonians depending only on the data $[x,y,a,b]$ of a representation.\end{theorem}

\begin{proof}   First, note that the quotient map corresponding to the reduction$$\left(T^*\mathcal F_{\underline{r^1}}(\alpha_1)\times\cdots\times T^*\mathcal F_{\underline{r^n}}(\alpha_n)\right)\times T^*\mathfrak{sl}(r,\mathbb C)^g\to\mathcal X^g_{\underline{r^1},\dots,\underline{r^n}}(\underline\alpha)$$is a Poisson morphism, and so the Gelfand-Tsetlin Poisson structures on the phase spaces $T^*\mathcal F_{\underline{r^i}}$ and the standard Lie-Poisson structure on $T^*\mathfrak{sl}(r,\mathbb C)^g$ descend to a well-defined one on $\mathcal X^g_{\underline{r^1},\dots,\underline{r^n}}(\underline\alpha)$, which we identify with the one arising from $\omega_I$. (Note that these are all incantations of the Lie-Poisson structure on $\mathfrak{sl}(r,\mathbb C)$, as the $T^*\mathcal F_{\underline{r^i}}$ are resolutions of closures of nilpotent orbits in $\mathfrak{sl}(r,\mathbb C)$, which are singular affine Poisson varieties).  In particular, the Hamiltonians on each summand descend to the quotient and Poisson commute with regards to the quotient Poisson structure.  What needs to be accounted for is dependency in the quotient Hamiltonians.  

Suppose the $i$-th arm is complete.  Then there is a sequence of trace-free matrices $(x^i_{k-1}y^i_{k-1})_0$ of size $k\times k$, $k=2,..,m$.  The moment map conditions (I)-(III) force the $k\times k$ matrix $(x^i_{k-1}y^i_{k-1})_0$ to be nilpotent of order $k$, and so can be put into a form where it has an upper (or lower) block of size $(k-1)\times(k-1)$ which in general is not nilpotent.  We denote this block by $b^i_{k-1}k$.  Each of these blocks contributes $k-1$ characteristic coefficients.  We think of these as traces $t_j$, where the trace $t_{k-1}$ is identified with the determinant of $b^i_{k-1}$ and so $t_1$ is the ordinary trace.  We can index these invariants as maps$$h^i_{j,k}:(x^i_{k-1}y^i_{k-1})_0\mapsto t_j(b^i_{k-1})$$for $j=1,\dots,k-1$.  These traces are a well-known complete set of complex-valued Hamiltonians for the Gelfand-Tsetlin integrable system for $T^*\mathcal F_{[r]}$.  The total number of generally nonzero functions of this form is $1+2+\dots+(r-1)=r(r-1)/2$, which is the complex dimension of $\mathcal F_{[r]}$, as expected.

When the $i$-th arm is minimal rather than complete, then the arm has associated to it a single $r\times r$ matrix $(x^i_{m_1-1}y^i_{m_i-1})_0=(x^i_1y^i_1)_0$ that is nilpotent of order $2$.  A nontrivial complex-valued Hamiltonian function is obtained from the entry in the top-right (or bottom-left) corner of this matrix.  This can be completed to a set of $r-1$ Poisson-commuting functions for the Gelfand-Tsetlin system on $T^*\mathcal F_{(1,r)}=T^*\mathbb P^{r-1}$, whose Hamiltonians we do not make explicit although their existence can be guaranteed in spite of the irregularity of the associated nilpotent orbit (see for instance \cite{Panyushev2019}).  In particular, we have exactly as many independent invariants as the complex dimension of $\mathcal F_{(1,r)}=\mathbb P^{r-1}$.

Lastly, the $j$-th copy of $T^*\mathfrak{sl}(r,\mathbb C)$ is coordinatized by position-momentum variables $(a^j,b^j)$. Prior to any reduction of this variety, the Hamiltonians are the $r^2-1$ independent entries $b^j_{p,q}$ of the matrices $b^j$ themselves.  As there are $g$ of these summands, we have a total of $g(r^2-1)$ invariants.

Hence, the quotient system on $\mathcal X^g_{\underline{r^1},\dots,\underline{r^n}}(\underline\alpha)$ has \emph{at most}$$cr(r-1)/2+(n-c)(r-1)+g(r^2-1)$$nontrivial functionally-independent Hamiltonians, where $c$ is the number of complete arms and $n-c$ is the number of minimal arms.  Now, consider the final reduction yielding $\mathcal X^g_{\underline{r^1},\dots,\underline{r^n}}(\underline\alpha)$ as a geometric quotient by $\mbox{SL}(r,\mathbb C)$.  This means that we take the complex quotient $\nu_\star^{-1}(0)/\mbox{SL}(r,\mathbb C)$.  We divide the remainder of the proof into two cases.

Consider $g=0$ first.  Note that the matrices $(x^i_{m_1-1}y^i_{m_i-1})_0$ along the first $i=1,\dots,r$ arms can all be fixed via the multiplication action of $\mbox{SL}(r,\mathbb C)$ on the central node.  (In order for the dimension of the moduli space to be nonnegative in this case, we need $n$ to be at least $r+1$ anyway.)  Moreover, the data $(x^{r+1}_{m_{r+1}-1}y^{r+1}_{m_{r+1}-1})_0$ along the $(r+1)$-th arm is dependent on the other arms by the complex hyperpolygon equations (I)-(III).  Hence, we subtract $(r+1)(r-1)=r^2-1$ from the upper-bound in each case to get the exact number of Hamiltonians.

We turn now to $g\geq1$.  Here, the multiplication action fixes the matrix $a^1\in\mbox{sl}(r,\mathbb C)$ and, subsequently, $b^1\in\mbox{sl}(r,\mathbb C)$ is completely determined by equation (I).   As a result, the invariants $b^1_{p,q}$ all become functionally dependent on the $x$ and $y$ data of the arms, which are untouched by the action.  Hence, we remove $r^2-1$ from the total.

As a final tally, the total number of independent Hamiltonians becomes$$cr(r-1)/2+(n-c)(r-1)+(g-1)(r^2-1).$$This is precisely the dimension of $\mathcal P^g_{\underline{r^1},\dots,\underline{r^n}}(\underline\alpha)$ when there are exactly $c$ complete arms and $n-c$ minimal ones, as per Proposition \eqref{PropDimP}.  We therefore have the desired maximal set of Poisson-commuting, functionally-independent Hamiltonians for the induced Poisson structure.\end{proof}

Immediately, we have:

\begin{corollary}\label{CorPresent} When the comet is complete, a maximal set of independent Gelfand-Tsetlin Hamiltonians is given by\begin{itemize}\item $h^i_{j,k}:(x^i_{k-1}y^i_{k-1})_0\mapsto t_j(b^i_{k-1})$ for $j=1,\dots,k-1$, $k=2,\dots,m$, $i=r+2,\dots n$ when $g=0$;\item $h^i_{j,k}:(x^i_{k-1}y^i_{k-1})_0\mapsto t_j(b^i_{k-1})$ for $j=1,\dots,k-1$, $k=2,\dots,m$, $i=1,\dots n$ when $g=1$;  \item $h^i_{j,k}:(x^i_{k-1}y^i_{k-1})_0\mapsto t_j(b^i_{k-1})$ for $j=1,\dots,k-1$, $k=2,\dots,m$, $i=1,\dots n$ and $b^j_{p,q}$ for $j=2,\dots,g$, $1\leq p,q\leq r$ (but omit $b^j_{r,r})$ when $g>1$.\end{itemize}\end{corollary} 

Finally, it is worth noting that, when there is at least one arm that is neither complete nor minimal, the combinatorics do not necessarily align in an obvious way to produce the correct half-dimensional result.
  
\section{Further directions}\label{final}

\subsection{Mirror symmetry and triple branes}

Here, we explore a different side of the physics of hyperpolygon spaces --- namely mirror symmetry.  Given any hyperk\"ahler variety with quaternions $I$, $J$ and $K$ and respective symplectic forms $\omega_I$, $\omega_J$ and $\omega_K$, we may ask how a given subvariety is compatible with those structures. Borrowing terminology from string theory, we refer to a Lagrangian subvariety with respect to one of the symplectic forms as an {\it $A$-brane} and to a complex subvariety with respect to one of the complex structures as a {\it $B$-brane}. Accordingly, when considering the whole hyperk\"ahler structure of the hyperk\"ahler variety one can seek subvarieties that are $(B,B,B)$, $(A,B,A)$, $(B,A,A)$, or $(A,A,B)$ with respect to the complex structures $(I,J,K)$ and the associated symplectic forms.  We call these ``triple branes'' in general.

For the moduli space of Higgs bundles without punctures, triple branes were considered in \cite{Kap}, where this moduli space is the target space for a topological sigma model.  Four-dimensional S-duality for this model corresponds to mirror symmetry between the modui space of $G$-Higgs bundles and the moduli space of $^{\text L}G$-Higgs bundles, where $^{\text L}G$ is the Langlands dual of a complex reductive group $G$.  In particular, triple branes in one moduli space are dual to ones in the mirror.  This observation has inspired the construction of many different types of branes in the Higgs moduli space through finite group actions and through holomorphic and anti-holomorphic involutions (e.g. see \cite{aba,slices,finite}). The involution technique has been adapted to the study of framed instantons in \cite{emilio1}, while the group-action technique has been used to construct triple branes in quiver varieties \cite{vic1,vic2}. 

On the one hand, there exists a convenient characterization for when $A$-branes and $B$-branes arise from involutions (cf. \cite{aba} for instance).  Given an analytic involution $\mathcal I$ on a nonsingular hyperk\"ahler variety, its fixed-point locus $\mathcal F$ is an $A$-brane with respect to the complex structure $I$ if $\mathcal I$ and $I$ anti-commute, that is, if $\mathcal II=-I\mathcal I$.  In other words, $\mathcal F$ is an $A$-brane if the involution is anti-holomorphic with respect to $I$.  On the other hand, $\mathcal F$ is a $B$-brane with respect to $I$ if $\mathcal I$ is holomorphic with respect to $I$, that is, if $\mathcal II=I\mathcal I$.  We can perform the same tests for $\mathcal I$ against $J$ and $K$.  If for each of $I,J,K$ we have that $\mathcal F$ is either $A$-type or $B$-type, then $\mathcal F$ is a triple brane described accordingly as one of $(B,B,B)$ or $(A,B,A)$ and so on.

On the other hand, the method in \cite{vic1,vic2} for producing branes in Nakajima quiver varieties uses quiver automorphisms $\sigma$ of $\mbox{Rep}(\overline{\mathcal Q})$ satisfying certain hypotheses:

\begin{itemize}
\item[(i)] for all $x^i_j$ and $y^i_j$, the automorphism satisfies $\sigma(y^i_j)=\sigma(x^i_j)^*$;
 \item[(ii)] either $\sigma(x^i_j)\in \{x^i_1,\ldots,x^i_{m_i-1}\}$ for all $i$ in which case it is {\it $\overline{\mathcal{Q}}$-symplectic}, or  $\sigma(x^i_j)\in \{y^i_1,\ldots,y^i_{m_{i}-1}\}$ for all $i,j$, in which case it is {\it $\overline{\mathcal{Q}}$-anti-symplectic}. 
\end{itemize}

Turning now to the specific hyperk\"ahler structure on $\mathcal X^g_{\underline{r^1},\dots,\underline{r^n}}(\underline\alpha)$, it is reasonable to ask about the existence of triple branes. We note here that triple branes within hyperpolygon spaces with $r=2$ and $g=0$ were constructed as examples in \cite{vic2} (and some of these were recently studied in some detail in \cite{godinho2019quasi}). To provide an additional example of a brane --- one that does not arise from a $\sigma$ of the type above and which is furthermore valid for an arbitrary comet\footnote{The partial flags allowed in generalized hyperpolygons could potentially fit within the framework of generalized $B$-opers introduced in \cite{genoper}, which yields branes that do not arises from involutions.} --- we consider the involution on $\mathcal X^g_{\underline{r^1},\dots,\underline{r^n}}(\underline\alpha)$ that negates cotangent directions:
\begin{eqnarray}\mathcal{I}_{-}:[x,y,a,b]\mapsto[x,-y,a,-b].\label{i1}\end{eqnarray}
The fixed point locus is given by the hyperpolygons of the form
$[x,0,a,0],$
which is precisely the generalized polygon space $\mathcal P^g_{\underline{r^1},\dots,\underline{r^n}}(\underline\alpha)$.

 \begin{proposition} \label{poly}The polygon space $\mathcal P^g_{\underline{r^1},\dots,\underline{r^n}}(\underline\alpha)$ is a $(B,A,A)$-brane within $\mathcal X^g_{\underline{r^1},\dots,\underline{r^n}}(\underline\alpha)$.
 \end{proposition}
\begin{proof} Consider the involution $\mathcal{I}_{-}$ in  \eqref{i1} and the complex structures $I,J,K$ on $\mathcal X^g_{\underline{r^1},\dots,\underline{r^n}}(\underline\alpha)$ defined in \eqref{i}-\eqref{k}.  At any hyperpolygon $[x,y,a,b]$, we can easily check how $\mathcal I_{-}$ behaves with respect to $I,J,K$.  We compute this for $J$ and $K$ here:

\begin{minipage}[t]{0.4\textwidth}

\begin{eqnarray}
J \mathcal{I}_{-}[x,y,a,b]&=&
J[x,-y,a,-b]\nonumber\\&=&
[y^*,x^*,b^*,a^*]\nonumber\\ &=&
\mathcal{I}_{-}[y^*,-x^*,b^*,-a^*]\nonumber\\&=&-\mathcal{I}_{-} J[x, y,a,b]\nonumber\end{eqnarray}

\end{minipage}\begin{minipage}[t]{0.6\textwidth}

\begin{eqnarray}
K \mathcal{I}_{-}[x,y,a,b]&=&K[x,-y,a,-b]\nonumber\\&=&[iy^*,ix^*,ib^*,ia^*]\nonumber\\&=&\mathcal{I}_{-}[iy^*,-ix^*,ib^*,-ia^*]\nonumber\\
&=&\mathcal{I}_{-}K [-x,   -y,-a,-b]\nonumber\\
&=&-\mathcal{I}_{-}K [x,y,a,b]\nonumber
\end{eqnarray}

\end{minipage}

\vspace{10pt}

In other words, $\mathcal I_{-}$ is anti-holomorphic with respect to each of $J$ and $K$.  A repetition of this calculation for $I$ reveals $\mathcal I_{-}$ to be holomorphic in that complex structure.  Thus, by our characterization above, the fixed point set $\mathcal P^g_{\underline{r^1},\dots,\underline{r^n}}(\underline\alpha)$ is a $(B,A,A)$-brane. \end{proof}
 
The reader may wish to compare this result with an analogous one of \cite[Section 3.2]{emilio1}, where a sign involution is studied on quiver varieties but whose fixed point set is indicated to be a $(B,B,B)$-brane. 
  
\subsection{Dualities between tame and wild hyperpolygons}

As an extension of the preceding constructions, we might also allow comet-shaped quivers in which two or more edges $e$ are permitted between two consecutive nodes along an arm (and hence two or more corresponding $-e$ arrows in the doubled quiver).  We describe this as a \emph{wild comet} and describe classes in the resulting quiver variety as \emph{wild hyperpolygons}, and examples of such quivers appear below in Figure \ref{maps23}.  From this point of view, our earlier hyperpolygons would be \emph{tame hyperpolygons} representing a \emph{tame comet}.  Wild hyperpolygons generalize similar objects arising from star-shaped quivers with multiple arrows between consecutive nodes in \cite{tbrane}.

To motivate one final observation, we wish to consider yet again the case where $\mathcal X^g_{\underline{r^1},\dots,\underline{r^n}}(\underline\alpha)$ arises from a tame comet in which each flag is the same --- for example, the complete and minimal comets.  One can note that there is a certain ambiguity here.  We can consider a hyperpolygon $[x,y,a,b]$ as being a represenation a tame comet \emph{or} from a wild comet with a single arm but with $n$-many $x$ arrows and $n$-many $y$ arrows connecting any two consecutive nodes. The wild comet comes about by identifying corresponding nodes of the arms.  We should note that the moment map equations for the tame comet are specializations of the ones for the wild comet and so there is a valid sense in which one quiver variety embeds into the other.  At the level of associated Higgs bundles, we are isolating a locus of wild Higgs bundles with an order-$n$ pole at infinity that is constructed from a tame Higgs bundle with $n$-many order-$1$ poles, simply by rearranging the residues.  The passage in and out of this locus, accomplished by means of quiver mutations, is closely related to degenerations of Painlev\'e equations.\\

\begin{figure}[!h]
\begin{center}
\includegraphics[width=1\linewidth]{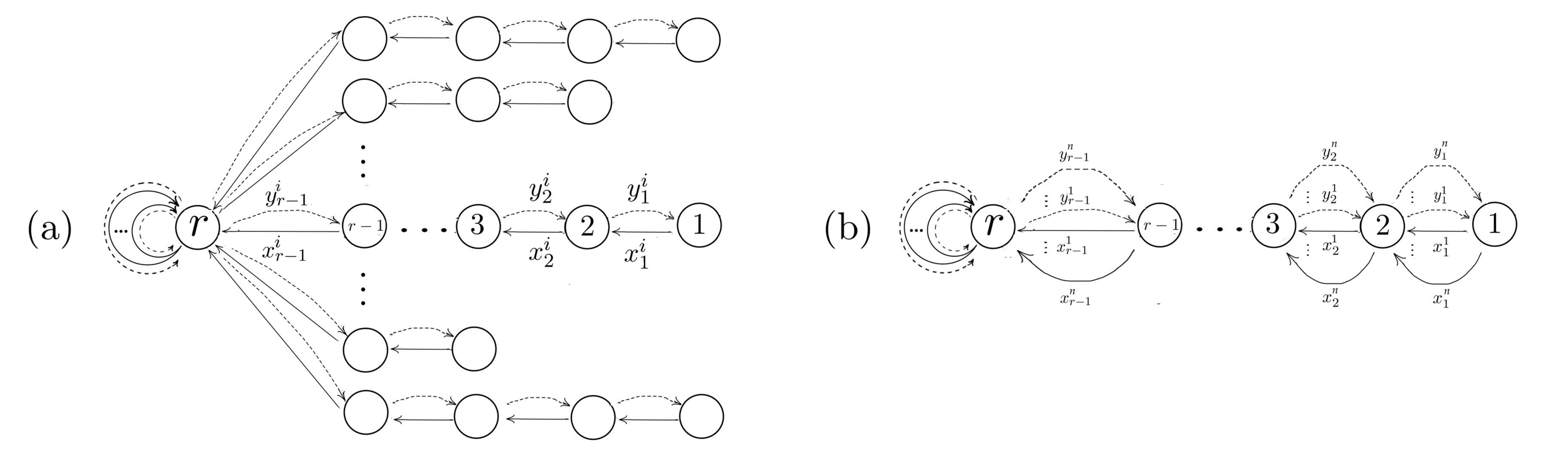}
\caption{(a) An example of a tame comet; (b) an example of a wild comet, whose pole order at the marked point is given by the number of the same direction between two consecutive nodes.}\label{maps23}
\end{center}
\end{figure}

%%%%%%%%%%%%%%%%
  
  \renewcommand{\baselinestretch}{1.1}
\renewcommand{\refname}{\bfseries{{ \quad References}}}

\bibliography{Hyperpolygons}{}
\bibliographystyle{acm}

\end{document}